\documentclass{amsart}
\usepackage{amsmath,amssymb,latexsym,amscd,color,array,enumerate}
\usepackage{tikz}
\usetikzlibrary{arrows,chains,matrix,positioning,scopes}
\usepackage[margin=1in]{geometry}
\usepackage[bookmarks=true,colorlinks=true, linkcolor=blue, citecolor=cyan]{hyperref}

\numberwithin{equation}{section}

\newtheorem{theorem}{Theorem}[section]
\newtheorem{proposition}[theorem]{Proposition}

\newtheorem{corollary}[theorem]{Corollary}

\newtheorem{definition}[theorem]{Definition}
\newtheorem{example}[theorem]{Example}
\newtheorem{lemma}[theorem]{Lemma}

\newtheorem{remark}[theorem]{Remark}

\theoremstyle{definition}

\newcommand{\cA}{\mathcal{A}}

\newcommand{\cC}{\mathcal{C}}

\newcommand{\cT}{\mathcal{T}}

\newcommand{\kk}{\Bbbk}
\newcommand{\ukk}{\underline\Bbbk}

\newcommand{\RR}{\mathbb{R}}

\newcommand{\ZZ}{\mathbb{Z}}

\newcommand{\partition}{\vdash}

\newcommand{\rsh}{\operatorname{rsh}}

\newcommand{\sh}{\operatorname{sh}}

\newcommand{\st}{\,:\,}
\newcommand{\supp}{\operatorname{supp}}

\newenvironment{enumeratea}{\begin{enumerate}[\upshape \quad(a)]}
                           {\end{enumerate}}
\newenvironment{enumeratei}{\begin{enumerate}[\upshape (i)]}
                           {\end{enumerate}}

\newcommand{\erase}[1]{{}}

\title{Greedy bases in rank 2 generalized cluster algebras}
\author{Dylan Rupel}
\address{\noindent Department of Mathematics, Northeastern University,
 Boston, MA 02115}
\email{d.rupel@neu.edu}
\date{\today}

\dedicatory{To the memory of Andrei Zelevinsky}

\makeatletter
\renewcommand{\@evenhead}{\tiny \thepage \hfill  D.~RUPEL \hfill}

\renewcommand{\@oddhead}{\tiny \hfill Rank 2 Generalized Greedy Bases
 \hfill \thepage}
\makeatother

\begin{document}
 \begin{abstract}
  In this note we extend the notion of greedy bases developed by Lee, Li, and Zelevinsky to rank two generalized cluster algebras, i.e. binomial exchange relations are replaced by polynomial exchange relations.  In the process we give a combinatorial construction in terms of a refined notion of compatible pairs on a maximal Dyck path.
 \end{abstract}
 \maketitle
 \tableofcontents

 \section{Introduction}\label{sec:intro}
  At the heart of the theory of cluster algebras is the hope for a combinatorially defined basis of a slightly complicated inductively defined algebra.  This is motivated by the conjectural relationship with the dual canonical basis: the cluster monomials should be contained in this basis.  Any such ``good basis" of a cluster algebra should satisfy two main conditions: it should be independent of the choice of an initial cluster and it should contain all cluster monomials.
  
  This note grew out of an attempt to understand one such ``good basis", namely the greedy basis, of a rank 2 cluster algebra.  This basis consists of indecomposable positive elements with a rich combinatorial description.  Our main goal in this project was to understand these combinatorics and to see in exactly what generality such a basis should exist.  To our surprise this basis is extremely general in the sense that it exists for all rank 2 generalized cluster algebras defined by arbitrary monic, palindromic polynomial exchange relations with non-negative integer coefficients.
  
  Our main theorem is the following analogue of the main result from \cite{lee-li-zelevinsky}, we refer the reader to the sections below for precise definitions.\\
  
  \begin{theorem}
  \label{th:main}
   Let $\kk$ denote an ordered field with positive cone $\Pi$ and suppose $P_1,P_2\in\Pi[z]$ are monic, palindromic polynomials.  Denote by $\ukk$ the $\ZZ$-subalgebra of $\kk$ generated by the coefficients of $P_1$ and $P_2$.  Write $\cA(P_1,P_2)$ for the associated generalized cluster algebra over $\ukk$.
   \begin{enumeratea}
    \item For each $(a_1,a_2)\in\ZZ^2$ there exists a (unique) greedy element $x[a_1,a_2]\in\cA(P_1,P_2)$.
    \item Each greedy element is an indecomposable positive element of $\cA(P_1,P_2)$.
    \item The greedy elements $x[a_1,a_2]$ for $(a_1,a_2)\in\ZZ^2$ form a $\ukk$-basis of the generalized cluster algebra $\cA(P_1,P_2)$, which we call the \emph{greedy basis}.
    \item The greedy basis is independent of the choice of initial cluster.
    \item The greedy basis contains all cluster monomials.
   \end{enumeratea}
  \end{theorem}
  
 \subsection{Organization} The paper is organized as follows.  Section~\ref{sec:generalized clusters} defines the generalized cluster algebra, establishes certain analogues of structural results from the classical theory of cluster algebras, and introduces the greedy basis of a rank 2 generalized cluster algebra.  This section also contains the proofs of our main results, some of them as corollaries of purely combinatorial statements contained in section~\ref{sec:compatible pairs}.  In section~\ref{sec:Dyck} we introduce the underlying structure for the combinatorics we seek to understand, namely maximal Dyck paths in a lattice rectangle of arbitrary size.  Section~\ref{sec:compatible pairs} develops the combinatorics of graded compatible pairs generalizing the combinatorics extensively studied in \cite{lee-li-zelevinsky}.  In the appendix, section~\ref{sec:multinomial}, we collect and prove auxiliary results related to multinomial coefficients necessary for working with arbitrary powers of polynomials.  
  
  The organization is chosen to allow the section on generalized cluster algebras to be read almost completely independently of the combinatorial sections.  In particular if one is willing to accept the technical results of section~\ref{sec:compatible pairs}, especially Proposition~\ref{prop:horizontal grading bijection} and Proposition~\ref{prop:supports}, then section~\ref{sec:generalized clusters} and the proof of the main theorem can be understood on their own.
  
 \subsection{Related and Future Work} In \cite{chekhov-shapiro} the authors use a special class of generalized cluster algebras to understand the Teichm\"uller theory of hyperbolic orbifold surfaces.  There they conjecture the positivity for the initial cluster expansion of all generalized cluster variables.  Theorem~\ref{th:main}, in particular parts (b) and (e), establishes the positivity of the initial cluster expansion of all generalized cluster variables, verifying their conjecture in the rank 2 case.  Analogous to the inductive proof of positivity for skew-symmetric cluster algebras of arbitrary rank given in \cite{lee-schiffler2}, it seems likely that one could prove the positivity for generalized cluster algebras of arbitrary rank inductively from our results presented here.
  
  Chekhov and Shapiro also offer a combinatorial description of their generalized cluster variables by passing to the universal cover of the orbifold surface and applying the combinatorics of $T$-paths \cite{schiffler} with an appropriate coloring.  In the case of a disk with two orbifold points (of arbitrary order) and one marked point on the boundary one obtains a rank 2 generalized cluster algebra.  In this case one may describe an explicit bijection between $T$-paths in the universal cover and our combinatorial theory of graded compatible pairs, perhaps this bijection will be expounded upon in a future work.
  
  One of the most powerful tools in the study of classical cluster algebras is the categorification using the representation theory of valued quivers, for rank 2 cluster algebras this is established in \cite{caldero-zelevinsky} and \cite{rupel1}.  Here the non-initial cluster monomials exactly correspond to rigid representations of the quiver.  It would be interesting to find a generalization of the representation theory of valued quivers for which rigid representations correspond to non-initial generalized cluster monomials in $\cA(P_1,P_2)$.  In this case it seems likely that the Caldero-Chapoton algebras defined in \cite{cerulli irelli-labardini fragoso-schroer} would be a useful tool.
  
  Non-commutative analogues of cluster recursions have been studied in \cite{lee-schiffler1, rupel2} and a combinatorial construction was given.  Polynomial generalizations have also been studied in \cite{usnich}, though only Laurentness is established and only for monic, palindromic polynomials $P_1=P_2$.  In a forthcoming work \cite{rupel3} we will use the combinatorics developed here to establish the Laurentness and positivity of rank 2 non-commutative generalized cluster variables.
 
  \subsection*{Acknowledgements} The author would like to thank Arkady Berenstein, Kyungyong Lee, Li Li, Thomas McConville, and Andrei Zelevinsky for useful discussions related to this project.  
  
  Many of these ideas were worked out while the author was a postdoctoral research fellow in the Cluster Algebras program at the Mathematical Sciences Research Institute.  The author would like to thank the MSRI for their hospitality and support.  The author would also like to thank the organizers of the Cluster Algebras program for giving him the opportunity to work in such a stimulating environment.

\section{Rank 2 Generalized Cluster Algebras and their Greedy Bases}\label{sec:generalized clusters}
  Fix any field $\kk$ of characteristic zero.  Let $P_1,P_2\in\kk[z]$ be arbitrary monic palindromic polynomials of degree $d_1$ and $d_2$ respectively, where a degree $d$ polynomial $P(z)$ is palindromic if $P(z)=z^dP(z^{-1})$.  Consider the ring $\kk(x_1,x_2)$ of rational functions in commuting variables $x_1$ and $x_2$.  We inductively define rational functions $x_k\in\kk(x_1,x_2)$ for $k\in\ZZ$ by the rule:
  \begin{equation}\label{eq:exchange relation}
   x_{k+1}x_{k-1}=\begin{cases}
                   P_1(x_k) & \text{ if $k$ is even;}\\
                   P_2(x_k) & \text{ if $k$ is odd.}
                  \end{cases}
  \end{equation}
  Write $\ukk$ for the $\ZZ$-subalgebra of $\kk$ generated by the coefficients of $P_1$ and $P_2$.  Define the \emph{generalized cluster algebra} $\cA(P_1,P_2)$ to be the $\ukk$-subalgebra of $\kk(x_1,x_2)$ generated by the set $\{x_k\}_{k\in\ZZ}$ of ``generalized cluster variables'' which we usually refer to simply as cluster variables.  For $k\in\ZZ$ the \emph{$k^{th}$ clusters} in $\cA(P_1,P_2)$ is the pair $\{x_k,x_{k+1}\}$ of neighboring cluster variables.
  
  \begin{example}
  \label{ex:generalized variables}
   We will take $P_1(z)=1+z+z^2$ and $P_2(z)=1+z+z^2+z^3$.  Then the first few cluster variables are given by
   \begin{align*}
    x_3&=x_1^{-1}(1+x_2+x_2^2);\\
    x_4&=x_1^{-3}x_2^{-1}[x_1^3+x_1^2(1+x_2+x_2^2)+x_1(1+x_2+x_2^2)^2+(1+x_2+x_2^2)^3];\\
    x_5&=x_1^{-5}x_2^{-2}[x_1^6+x_1^5(2+x_2)+x_1^4(3+4x_2+4x_2^2+x_2^3)+x_1^3(4+9x_2+14x_2^2+11x_2^3+6x_2^4+x_2^5)\\
    &\quad+3x_1^2(1+x_2+x_2^2)^3+2x_1(1+x_2+x_2^2)^4+(1+x_2+x_2^2)^5].
   \end{align*}
  \end{example}
  
  \subsection{Structure of Rank 2 Generalized Cluster Algebras}  It will be convenient to introduce Laurent polynomial subalgebras $\cT_k=\ukk[x_k^{\pm1},x_{k+1}^{\pm1}]\subset\kk(x_1,x_2)$ for each $k\in\ZZ$.  Then we may formulate our first structural result on rank 2 generalized cluster algebras.
  \begin{theorem}[Laurent Phenomenon]
  \label{th:Laurent}
   The generalized cluster algebra $\cA(P_1,P_2)$ is contained in $\cT_k$ for each $k\in\ZZ$.
  \end{theorem}
  Before presenting the proof we need to introduce some additional notation.  Write 
  \[P_1(z)=\rho_0+\rho_1z+\cdots+\rho_{d_1}z^{d_1}\quad\text{ and }\quad P_2(z)=\varrho_0+\varrho_1z+\cdots+\varrho_{d_2}z^{d_2},\]
  since these are monic and palindromic the coefficients satisfy:
  \begin{align*}
   &\rho_0=\rho_{d_1}=1=\varrho_0=\varrho_{d_2};\\
   &\rho_t=\rho_{d_1-t}\quad\text{for $0\le t\le d_1$;}\\
   &\varrho_t=\varrho_{d_2-t}\quad\text{for $0\le t\le d_2$.}
  \end{align*}
  We learned the following proof of Laurentness for rank 2 cluster algebras from Arkady Berenstein.
  \begin{proof}
   We will prove for each $m\in\ZZ$ the containment $x_m\in\cT_k$.  This is accomplished by induction by considering the element $x_{m+1}^{d_1}x_{m+4}$ as follows (without loss of generality assume $m$ is odd):
   \begin{align*}
    x_{m+1}^{d_1}x_{m+4}
    &=\frac{x_{m+1}^{d_1}P_1(x_{m+3})}{x_{m+2}}=\frac{x_{m+1}^{d_1}P_1(x_{m+3})-P_1(x_{m+1})}{x_{m+2}}+\frac{P_1(x_{m+1})}{x_{m+2}}\\
    &=\frac{\sum\limits_{t=0}^{d_1}(\rho_tx_{m+1}^{d_1}x_{m+3}^t-\rho_{d_1-t}x_{m+1}^{d_1-t})}{x_{m+2}}+x_m=\frac{\sum\limits_{t=1}^{d_1}\rho_t(x_{m+1}^tx_{m+3}^t-1)x_{m+1}^{d_1-t}}{x_{m+2}}+x_m,
   \end{align*}
   where the last equality used that $P_1$ was palindromic.  Using that $P_2$ was assumed monic, we see that $x_{m+1}^tx_{m+3}^t-1=P^t_2(x_{m+2})-1\in x_{m+2}\ukk[x_{m+2}]$ for each $t\ge0$ and thus
   \[\frac{\sum\limits_{t=1}^{d_1}\rho_t(x_{m+1}^tx_{m+3}^t-1)x_{m+1}^{d_1-t}}{x_{m+2}}\in\ukk[x_{m+1},x_{m+2}].\]
   It follows that $x_m\in\ukk[x_{m+1},x_{m+2},x_{m+3},x_{m+4}]$ for each $m\in\ZZ$.  By a similar calculation one may prove the membership $x_m\in\ukk[x_{m-4},x_{m-3},x_{m-2},x_{m-1}]$.  By inductively ``shifting the viewing window" we see that $x_m\in\ukk[x_{k-1},x_k,x_{k+1},x_{k+2}]\subset\cT_k$ for each $k\in\ZZ$.
  \end{proof}
  \begin{remark}
   If $\kk$ is an ordered field it makes sense to talk about positive elements, we will make this precise in Section~\ref{sec:greedy}.  In this case, when $P_1$ and $P_2$ have positive coefficients it is not apparent from this proof of Laurentness that the generalized cluster variables can be expressed as Laurent polynomials with positive coefficients.  Establishing this positivity will be one of the central results of this note.
  \end{remark}
   
   It follows from the proof of Theorem~\ref{th:Laurent} that the rank 2 generalized cluster algebra $\cA(P_1,P_2)$ is equal to each of its \emph{lower bound algebras} $\ukk[x_{k-1},x_k,x_{k+1},x_{k+2}]$ for $k\in\ZZ$.  This allows us to identify a relatively simple $\ukk$-basis of $\cA(P_1,P_2)$, indeed for each $(a_1,a_2)\in\ZZ^2$ we define the standard monomial $z_k[a_1,a_2]\in\cA(P_1,P_2)$ in the $k^{th}$ cluster by
   \begin{equation*}
    z_k[a_1,a_2]=x_{k-1}^{[a_2]_+}x_k^{[-a_1]_+}x_{k+1}^{[-a_2]_+}x_{k+2}^{[a_1]_+},
   \end{equation*}
   where we write $[a]_+=\max(a,0)$.  We define the \emph{cluster monomials} to be the set $\big\{z_k[a_1,a_2]\big\}_{a_1,a_2\in\ZZ_{\le0},k\in\ZZ}$
   \begin{theorem}
   \label{th:standard monomial basis} 
    For each $k\in\ZZ$ the set of all standard monomials $\big\{z_k[a_1,a_2]\big\}_{a_1,a_2\in\ZZ}$ forms a $\ukk$-basis of $\cA(P_1,P_2)$.
   \end{theorem}
   \begin{proof}
    Since $\cA(P_1,P_2)=\ukk[x_{k-1},x_k,x_{k+1},x_{k+2}]$ we know that $\cA(P_1,P_2)$ is spanned over $\ukk$ by all monomials of the form $x_{k-1}^{b_1}x_k^{b_2}x_{k+1}^{b_3}x_{k+2}^{b_4}$ where each $b_i$ is a nonnegative integer.  Using the exchange relations \eqref{eq:exchange relation} we may eliminate any factors of the form $x_{k-1}x_{k+1}$ or $x_kx_{k+2}$, in particular we see that the standard monomials $\big\{z_k[a_1,a_2]\big\}_{a_1,a_2\in\ZZ}$ span $\cA(P_1,P_2)$.  
    
    It only remains to show that the set of standard monomials is linearly independent over $\ukk$.  For this we note that the smallest monomial appearing in the $k^{th}$ cluster expansion of $z_k[a_1,a_2]$ is $x_k^{-a_1}x_{k+1}^{-a_2}$, but the set of monomials $\{x_k^{-a_1}x_{k+1}^{-a_2}\}_{a_1,a_2\in\ZZ}$ is linearly independent over $\ukk$ in $\cT_k$.  It follows that the standard monomials are linearly independent in $\cA(P_1,P_2)\subset\cT_k$.
   \end{proof}
  
  In fact there is a stronger version of the Laurent phenomenon which states that $\cA(P_1,P_2)$ is exactly the set of all universally Laurent elements of $\ukk(x_1,x_2)$, i.e. $\cA(P_1,P_2)$ consists of all rational functions which are Laurent when expressed in terms of any given cluster $\{x_k,x_{k+1}\}$.  Moreover, as the next result states, the check for membership in $\cA(P_1,P_2)$ can be restricted to checking Laurentness for any three consecutive clusters.
  \begin{theorem}[Strong Laurent Phenomenon]
  \label{th:strong Laurent}
   For any $m\in\ZZ$ we have
   \begin{equation}
    \cA(P_1,P_2)=\bigcap_{k\in\ZZ}\cT_k=\bigcap_{k=m-1}^{m+1}\cT_k.
   \end{equation}
  \end{theorem}
  \begin{proof}
   To prove this we establish an analogue of the equality of upper and lower bounds for acyclic cluster algebras from \cite{bfz}.  More precisely, the result will follow if we can establish the equality
   \begin{equation}
   \label{eq:strong Laurent}
    \ukk[x_{m-1},x_m,x_{m+1},x_{m+2}]=\cT_{m-1}\cap\cT_m\cap\cT_{m+1}.
   \end{equation}
   Indeed, $\cA(P_1,P_2)$ has already been identified with the left hand side in the course of proving Theorem~\ref{th:Laurent}.  
   
   The containment of the left hand side of \eqref{eq:strong Laurent} into the right hand side follows from the Laurent Phenomenon.  Our goal is to establish the opposite containment.  We accomplish this by decomposing $\cT_m$ by decomposing
   \[\cT_m=\ukk[x_m^{-1},x_{m+1}^{-1}]+\ukk[x_m,x_{m+1}^{\pm1}]+\ukk[x_m^{\pm1},x_{m+1}]\]
   and proving the containment upon intersection of each summand with $\cT_{m-1}\cap\cT_{m+1}$.
   
   Without loss of generality we will assume $m$ is even.  We begin by establishing the equality
   \begin{equation}
   \label{eq:intersections1}
    \cT_{m-1}\cap\ukk[x_m^{-1},x_{m+1}^{-1}]\cap\cT_{m+1}=\begin{cases} \ukk & \text{ if $2\le d_1d_2$;}\\ \ukk[x_{m-1}x_{m+2}] & \text{ if $d_1=d_2=1$;}\\ \ukk[x_{m-1},x_{m-1}^{d_2}x_{m+2}] & \text{ if $d_1=0$;}\\ \ukk[x_{m-1}x_{m+2}^{d_1},x_{m+2}] & \text{ if $d_2=0$.}\end{cases}
   \end{equation}
   Notice that each ring on the right hand side of \eqref{eq:intersections1} is contained in the left hand side.  Indeed, by the Laurent Phenomenon each of these rings is contained in $\cT_{m-1}\cap\cT_{m+1}$.  When $d_1=d_2=1$ we have 
   \[x_{m-1}x_{m+2}=(x_mx_{m+1}^{-1}+x_{m+1}^{-1})(x_m^{-1}x_{m+1}+x_m^{-1})=1+x_m^{-1}+x_{m+1}^{-1}+x_m^{-1}x_{m+1}^{-1}\in\ukk[x_m^{-1},x_{m+1}^{-1}].\]
   When $d_1=0$ we have $x_{m-1}=x_{m+1}^{-1}$ and 
   \[x_{m-1}^{d_2}x_{m+2}=x_{m+1}^{-d_2}P_2(x_{m+1})x_m^{-1}=P_2(x_{m+1}^{-1})x_m^{-1}\in\ukk[x_m^{-1},x_{m+1}^{-1}],\]
   where we used that $P_2$ is palindromic.  Finally when $d_2=0$ we have $x_{m+2}=x_m^{-1}$ and 
   \[x_{m-1}x_{m+2}^{d_1}=x_{m+1}^{-1}P_1(x_m)x_m^{-d_1}=x_{m+1}^{-1}P_1(x_m^{-1})\in\ukk[x_m^{-1},x_{m+1}^{-1}].\]
   
   Now we check in each case that the intersection on the left hand side of \eqref{eq:intersections1} is contained in the desired ring on the right.  Consider an element $y\in\ukk[x_m^{-1},x_{m+1}^{-1}]$, then $y$ may be written in the form
   \begin{equation}
   \label{eq:negative exponents}
    y=\sum\limits_{s,t\ge0}c_{s,t}x_m^{-s}x_{m+1}^{-t},
   \end{equation}
   where $c_{s,t}\in\ukk$ is zero for all but finitely many $s$ and $t$.  Applying the identity $x_{m+1}=\frac{P_1(x_m)}{x_{m-1}}$ we get
   \[y=\sum\limits_{t\ge0}\frac{\sum\limits_{s\ge0}c_{s,t}x_m^{-s}}{P^t_1(x_m)}x_{m-1}^t=\sum\limits_{t\ge0}\frac{\sum\limits_{s\ge0}c_{s,t}x_m^{-s}}{P^t_1(x_m^{-1})}x_m^{-d_1t}x_{m-1}^t,\]
   where the second equality used that $P_1$ is palindromic.  It follows from the membership $y\in\cT_{m-1}$ that $\frac{\sum\limits_{s\ge0}c_{s,t}z^s}{P_1^t(z)}\in\ukk[z]$ for all $t\ge0$.  Denote this polynomial by $C_t$, i.e. 
   \begin{equation}
   \label{eq:degree bound1}
    y=\sum\limits_{t\ge0}C_t(x_m^{-1})x_m^{-d_1t}x_{m-1}^t=\sum\limits_{t\ge0}C_t(x_m^{-1})P^t_1(x_m)x_m^{-d_1t}x_{m+1}^{-t}=\sum\limits_{t\ge0}C_t(x_m^{-1})P^t_1(x_m^{-1})x_{m+1}^{-t}.
   \end{equation}
     By a similar calculation with \eqref{eq:negative exponents} and the identity $x_m=\frac{P_2(x_{m+1})}{x_{m+2}}$ we get 
     \begin{equation}
     \label{eq:degree bound2}
      y=\sum\limits_{s\ge0}C'_s(x_{m+1}^{-1})x_{m+1}^{-d_2s}x_{m+2}^s=\sum\limits_{s\ge0}C'_s(x_{m+1}^{-1})P^s_2(x_{m+1}^{-1})x_m^{-s}
     \end{equation}
     for some polynomials $C'_s\in\ukk[z]$.
     
     If $d_2=0$, we have $x_m^{-1}=x_{m+2}$ and the first equality in \eqref{eq:degree bound1} becomes
     \[y=\sum\limits_{t\ge0}C_t(x_{m+2})x_{m+2}^{d_1t}x_{m-1}^t\in\ukk[x_{m-1}x_{m+2}^{d_1},x_{m+2}].\]
     Similarly, if $d_1=0$ we have $x_{m+1}^{-1}=x_{m-1}$ and the first equality in \eqref{eq:degree bound2} becomes
     \[y=\sum\limits_{s\ge0}C'_s(x_{m-1})x_{m-1}^{d_2s}x_{m+2}^s\in\ukk[x_{m-1},x_{m-1}^{d_2}x_{m+2}].\]
     
     Now we consider the cases $d_1d_2\ne0$, first we introduce some more notation.  Denote by $\sigma(t)$ the largest index so that $c_{\sigma(t),t}\ne0$ in the expansion \eqref{eq:negative exponents}, then using the last equality in \eqref{eq:degree bound1} we must have $\sigma(t)\ge d_1t$ for each $t\ge0$.  Similarly denote by $\tau(s)$ the largest index so that $c_{s,\tau(s)}\ne0$ in \eqref{eq:negative exponents}, then the last equality in \eqref{eq:degree bound2} gives $\tau(s)\ge d_2s$ for each $s\ge0$.  
     
   Suppose $t$ is the largest so that $c_{s,t}\ne0$ for some $s$.  Then $c_{\sigma(t),\tau(\sigma(t))}\ne0$ and $\tau\big(\sigma(t)\big)\ge d_2\sigma(t)\ge d_1d_2t$.  For $d_1d_2\ge2$ and $t>0$, this implies $\tau\big(\sigma(t)\big)>t$ contradicting the maximality of $t$.  Thus we must have $t=0$, by a similar argument we see that $s=0$ when $d_1d_2\ge2$.  It follows that $y\in\ukk$.
   
   Now suppose $d_1d_2=1$.  Then we have $\tau\big(\sigma(t)\big)\ge\sigma(t)\ge t$.  Since $t$ was maximal each of these inequalities must be an equality.  We establish the membership $y\in\ukk[x_{m-1}x_{m+2}]$ by a simple induction on the maximal value $t$ such that $c_{t,t}\ne0$.  Indeed, when $t=0$ the claim is clear.  When $t>0$, the element $y-c_{t,t}x^t_{m-1}x^t_{m+2}$ is contained in $\ukk[x_{m-1}x_{m+2}]$ by the induction hypothesis, it follows that $y\in\ukk[x_{m-1}x_{m+2}]$.  This completes the proof of the equality \eqref{eq:intersections1}.
   
   The final step in the proof of Theorem~\ref{th:strong Laurent} is to show the equality
   \[\cT_{m-1}\cap\ukk[x_m,x_{m+1}^{\pm1}]=\ukk[x_{m-1},x_m,x_{m+1}]\subset\cT_{m+1}.\]
   The inclusion ``$\supseteq$" is clear, thus we aim to establish the inclusion ``$\subseteq$".  Indeed, any $y\in\ukk[x_m,x_{m+1}^{\pm1}]$ can be written in the form
   \[y=\sum\limits_{t=-N}^Nc_tx_{m+1}^t,\]
   for some positive integer $N$ where $c_t\in\ukk[x_m]$ for each $t$.  Making the substitution $x_{m+1}=\frac{P_1(x_m)}{x_{m-1}}$ we get 
   \[y=\sum\limits_{t=0}^Nc_tP_1^t(x_m)x_{m-1}^{-t}+\sum\limits_{t=1}^N\frac{c_{-t}}{P_1^t(x_m)}x_{m-1}^t.\]
   If in addition we assume $y\in\cT_{m-1}$, we must have $\frac{c_{-t}}{P_1^t(x_m)}\in\ukk[x_m^{\pm1}]$ for $1\le t\le N$.  But $P_1(x_m)$ has constant term 1, and so we must actually have $\frac{c_{-t}}{P_1^t(x_m)}\in\ukk[x_m]$ for $1\le t\le N$.  Thus we get the containment
   \[y=\sum\limits_{t=0}^Nc_tx_{m+1}^t+\sum\limits_{t=1}^N\frac{c_{-t}}{P_1^t(x_m)}x_{m-1}^t\in\ukk[x_{m-1},x_m,x_{m+1}]\]
   as desired.  By a similar calculation we get $\ukk[x_m^{\pm1},x_{m+1}]\cap\cT_{m+1}=\ukk[x_m,x_{m+1},x_{m+2}]\subset\cT_{m-1}$.\\
   
   \noindent We are now ready to complete the proof:
   \begin{align*}
    \cT_{m-1}\cap\cT_m\cap\cT_{m+1}
    &=\cT_{m-1}\cap\big(\ukk[x_m^{-1},x_{m+1}^{-1}]+\ukk[x_m,x_{m+1}^{\pm1}]+\ukk[x_m^{\pm1},x_{m+1}]\big)\cap\cT_{m+1}\\
    &=\cT_{m-1}\cap\ukk[x_m^{-1},x_{m+1}^{-1}]\cap\cT_{m+1}+\ukk[x_{m-1},x_m,x_{m+1}]+\ukk[x_m,x_{m+1},x_{m+2}]\\
    &\subset\ukk[x_{m-1},x_m,x_{m+1},x_{m+2}],
   \end{align*}   
   as desired.  This complete the proof of Theorem~\ref{th:strong Laurent}.
  \end{proof}
  
  \subsection{Greedy Elements in Rank 2 Generalized Cluster Algebras}\label{sec:greedy}  In this section we introduce greedy elements in a rank 2 generalized cluster algebra.  For this we need to work over a field with an inherent notion of positivity.
  
 \begin{definition}
  A \emph{prepostive cone} $\Pi\subset\kk\setminus\{-1\}$ satisfies the following closure properties:
  \begin{itemize}
   \item For any $a,b\in\Pi$, we have $a+b,ab\in\Pi$;
   \item For any $a\in\kk$, we have $a^2\in\Pi$.
  \end{itemize}
  A prepositive cone is called \emph{positive} if $\kk=\Pi\cup-\Pi$.  A field $\kk$ together with a positive cone $\Pi$ is called an \emph{ordered field}.  We define a total order $\le$ on $\kk$ by declaring for $a,b\in\kk$ that $a\le b$ if $b-a\in\Pi$.
 \end{definition}
  
  From now on we assume $\kk$ is ordered with positive cone $\Pi$ and $P_1,P_2\in\Pi[z]$.  Define the semiring $\ukk_{\ge0}=\ukk\cap\Pi$, this is the $\ZZ_{\ge0}$-subsemiring of $\kk$ generated by the coefficients of $P_1$ and $P_2$.  For each $k\in\ZZ$ define the positive semiring $\cT_k^{\ge0}=\ukk_{\ge0}[x_k^{\pm1},x_{k+1}^{\pm1}]\subset\cT_k$, the set of \emph{positive} elements of $\cT_k$.  We are interested in those elements of $\cA(P_1,P_2)$ which are universally positive in the following sense.
  \begin{definition}
   An nonzero element $x\in\cA(P_1,P_2)$ is called \emph{universally positive} if its expression as a Laurent polynomial in any cluster $\{x_k,x_{k+1}\}$ only contains nonnegative coefficients, i.e. $x\in\bigcap_{k\in\ZZ}\cT_k^{\ge0}$.  A positive element is called \emph{indecomposable} if it cannot be written as a sum of two nonzero positive elements.
  \end{definition}
  Our aim in this note is to establish the existence of a universally positive indecomposable basis of $\cA(P_1,P_2)$ called the \emph{greedy basis}.  To describe this basis we need to introduce the following notion of a pointed element of $\cT_k$.
  \begin{definition}
   Let $(a_1,a_2)\in\ZZ^2$.  An element $x\in\cT_k$ is said to be \emph{pointed at $(a_1,a_2)$ in the $k^{th}$ cluster} if it can be written in the form
  \[x=x_k^{-a_1}x_{k+1}^{-a_2}\sum\limits_{p,q\ge0}c(p,q)x_k^px_{k+1}^q,\]
  where $c(p,q)\in\ukk$ and $c(0,0)=1$.
  \end{definition}
  \noindent We will simply refer to an element as ``pointed" if it is pointed in the initial cluster $\{x_1,x_2\}$.  
  
  Consider a collection $\big\{x[a_1,a_2]\big\}_{a_1,a_2\in\ZZ}$ where $x[a_1,a_2]$ is pointed at $(a_1,a_2)\in\ZZ^2$.  Call such a collection of pointed elements ``bounded" if $x[a_1,a_2]=z_1[a_1,a_2]$ is a standard monomial for $(a_1,a_2)\in\ZZ^2\setminus\ZZ_{>0}^2$.  The next result claims that any complete bounded collection of pointed elements in $\cA(P_1,P_2)$ forms a $\ukk$-basis.
  \begin{proposition}
  \label{prop:bounded basis}
   Suppose there exists a complete bounded collection $\big\{x[a_1,a_2]\big\}_{a_1,a_2\in\ZZ}\subset\cA(P_1,P_2)$ where $x[a_1,a_2]$ is pointed at $(a_1,a_2)\in\ZZ^2$.  Then $\big\{x[a_1,a_2]\big\}_{a_1,a_2\in\ZZ}$ forms a $\ukk$-basis of $\cA(P_1,P_2)$.
  \end{proposition}
  \begin{proof}
   By definition the smallest monomial appearing in a pointed element $x[a_1,a_2]$ is $x_1^{-a_1}x_2^{-a_2}$, but the monomials $\{x_k^{-a_1}x_{k+1}^{-a_2}\}_{a_1,a_2\in\ZZ}$ are linearly independent over $\ukk$ in $\cT_1$.  It follows that the set $\{x[a_1,a_2]\}_{a_1,a_2\in\ZZ}$ is linearly independent in $	\cA(P_1,P_2)\subset\cT_1$.  
   
   To finish we only need to show that these pointed elements span $\cA(P_1,P_2)$.  Suppose $x\in\cA(P_1,P_2)$ and write $x=\sum\limits_{b_1,b_2\in\ZZ} c_{b_1,b_2}x_1^{-b_1}x_2^{-b_2}$ for some $c_{b_1,b_2}\in\ZZ$.  Define the ``negative support" $S(x)$ of $x$ by 
   \[S(x)=\{(b_1,b_2)\in\ZZ_{>0}^2\st c_{b_1,b_2}\ne0\}.\]
   Moreover, write $m(x)=\max\{b_1+b_2\st (b_1,b_2)\in S(x)\}\cup\{0\}$ and let $M(x)\subset S(x)$ denote those pairs $(b_1,b_2)\in S(x)$ for which $b_1+b_2=m(x)$ is maximal.  Notice that the element
   \[x'=x-\sum\limits_{(b_1,b_1)\in M(x)} c_{b_1,b_2}x[b_1,b_2]\]
   satisfies $m(x')<m(x)$.  Iterating this process of subtracting off maximal points we may produce an element $x''\in\cA(P_1,P_2)$ with $m(x'')=0$, i.e. $S(x'')=\emptyset$.  Then by Theorem~\ref{th:standard monomial basis} and using that the collection $\big\{x[a_1,a_2]\big\}_{a_1,a_2\in\ZZ}$ is bounded we may write 
   \[x''=\sum\limits_{(a_1,a_2)\in\ZZ^2\setminus\ZZ_{>0}^2}d_{a_1,a_2}z_1[a_1,a_2]=\sum\limits_{(a_1,a_2)\in\ZZ^2\setminus\ZZ_{>0}^2}d_{a_1,a_2}x[a_1,a_2],\quad\quad \big(d_{a_1,a_2}\in\ukk\big)\]
   where we note that $z_1[b_1,b_2]$ for $(b_1,b_2)\in\ZZ_{>0}^2$ does not appear by the construction of $x''$.  It follows that $x''$, and hence $x$, is contained in the $\ukk$-span of $\big\{x[a_1,a_2]\big\}_{a_1,a_2\in\ZZ}$.  Thus we see that the collection $\big\{x[a_1,a_2]\big\}_{a_1,a_2\in\ZZ}$ spans $\cA(P_1,P_2)$.
  \end{proof}
  
  In view of Theorem~\ref{th:strong Laurent}, it is natural to look at elements that are positive in three consecutive clusters.  The following result gives precise conditions on the coefficients $c(p,q)$ for a pointed element to be positive in the initial cluster and its two immediate neighbors.  The key insight of \cite{lee-li-zelevinsky} is that knowing this restricted positivity can be enough to know an element is universally positive, but more on that later.  Though this result holds in greater generality we restrict attention to elements pointed in the initial cluster.
  \begin{proposition}\label{prop:positive elements}
   Suppose $x\in\cT_1$ is pointed at $(a_1,a_2)\in\ZZ^2$ and $x\in\cT_0^{\ge0}\cap\cT_1^{\ge0}\cap\cT_2^{\ge0}$ is positive when expanded in each of three consecutive clusters.  Then the coefficients $c(p,q)$ in the initial cluster expansion of $x$ satisfy the following inequality in $\kk$ for all $(p,q)\in\ZZ^2_{\ge0}$:
   \begin{align*}
    c(p,q)\ge
    &\max\Bigg(\Bigg[\sum\limits_{k=1}^p\sum\limits_{(k_1,\ldots,k_{d_2})\partition k}\!\!\!\!\! (-1)^{k-1} c(p-k_1-2k_2-\cdots-d_2k_{d_2},q) \varrho_1^{k_1}\cdots \varrho_{d_2}^{k_{d_2}}{a_2-q+k-1\choose a_2-q-1,k_1,\ldots,k_{d_2}}\Bigg]_+,\\
    &\quad\quad\quad\Bigg[\sum\limits_{\ell=1}^q\sum\limits_{(\ell_1,\ldots,\ell_{d_1})\partition \ell}\!\!\!\!\! (-1)^{\ell-1} c(p,q-\ell_1-2\ell_2-\cdots-d_1\ell_{d_1}) \rho_1^{\ell_1}\cdots \rho_{d_1}^{\ell_{d_1}}{a_1-p+\ell-1\choose a_1-p-1,\ell_1,\ldots,\ell_{d_1}}\Bigg]_+\Bigg),
   \end{align*}
   where $\rho_i$ and $\varrho_i$ denote the coefficients of the polynomials $P_1$ and $P_2$ respectively and where $[a]_+=\max(a,0)$.
  \end{proposition}
  Before proceeding with the proof we refer the reader to the appendix, section~\ref{sec:multinomial}, for our notations and conventions related to multinomial coefficients.
  \begin{proof}
   Consider the initial cluster expansion of $x$:
   \begin{equation}
   \label{eq:x pointed expansion}
    x=x_1^{-a_1}x_2^{-a_2}\sum\limits_{p,q\ge0}c(p,q)x_1^px_2^q.
   \end{equation}
   To establish the second inequality we will make the substitution $x_1=\frac{P_1(x_2)}{x_3}$ in \eqref{eq:x pointed expansion}, we leave the substitution $x_2=\frac{P_2(x_1)}{x_0}$ in \eqref{eq:x pointed expansion}, and thus the verification of the first inequality, to the reader.  Applying Corollary~\ref{cor:polynomial negative power} we may expand $x_1^{p-a_1}=\left(\frac{P_1(x_2)}{x_3}\right)^{p-a_1}$ as
   \[x_1^{p-a_1}=\sum\limits_{\ell\ge0}\sum\limits_{(\ell_1,\ldots,\ell_{d_1})\partition \ell} (-1)^\ell \rho_1^{\ell_1}\cdots \rho_{d_1}^{\ell_{d_1}}{a_1-p+\ell-1\choose a_1-p-1,\ell_1,\ldots, \ell_{d_1}} x_2^{\ell_1+2\ell_2+\cdots+d_1\ell_{d_1}}x_3^{a_1-p},\]
   where we used that $\rho_0=1$ to slightly simplify the expression.  Substituting into the initial cluster expansion \eqref{eq:x pointed expansion} of $x$, we see that the coefficient of $x_2^{q-a_2}x_3^{a_1-p}$ in the expansion of $x$ as an element of $\cT_2$ is given by
   \begin{equation}
   \label{eq:expansion coefficient}
    \sum\limits_{\ell\ge0}\sum\limits_{(\ell_1,\ldots,\ell_{d_1})\partition \ell} (-1)^\ell c(p,q-\ell_1-2\ell_2-\cdots-d_1\ell_{d_1}) \rho_1^{\ell_1}\cdots \rho_{d_1}^{\ell_{d_1}}{a_1-p+\ell-1\choose a_1-p-1,\ell_1,\ldots,\ell_{d_1}}\ge0,
   \end{equation}
   which must be non-negative by the membership $x\in\cT_2^{\ge0}$.  Note that the summand is zero for any $\ell>q$ and is equal to $c(p,q)$ for $\ell=0$.  Solving the inequality \eqref{eq:expansion coefficient} for $c(p,q)$ and remembering that $c(p,q)\ge0$ completes the proof.
  \end{proof}
   
  We will be interested in certain \emph{greedy} pointed elements $x[a_1,a_2]\in\cT_1$ for $(a_1,a_2)\in\ZZ^2$ which are ``minimally positive" in the sense that the bound in Proposition~\ref{prop:positive elements} is sharp.
  \begin{definition}
   For $(a_1,a_2)\in\ZZ^2$ define the greedy pointed element $x[a_1,a_2]\in\cT_1$ whose pointed expansion coefficients $c(p,q)$ are given by the following ``greedy recursion":
   \begin{align}
    \nonumber c(p,q)=
    &\max\Bigg(\Bigg[\sum\limits_{k=1}^p\sum\limits_{(k_1,\ldots,k_{d_2})\partition k}\!\!\!\!\! (-1)^{k-1} c(p-k_1-2k_2-\cdots-d_2k_{d_2},q) \varrho_1^{k_1}\cdots \varrho_{d_2}^{k_{d_2}}{a_2-q+k-1\choose a_2-q-1,k_1,\ldots,k_{d_2}}\Bigg]_+,\\
    \label{eq:greedy recursion max}&\quad\quad\quad\Bigg[\sum\limits_{\ell=1}^q\sum\limits_{(\ell_1,\ldots,\ell_{d_1})\partition \ell}\!\!\!\!\! (-1)^{\ell-1} c(p,q-\ell_1-2\ell_2-\cdots-d_1\ell_{d_1}) \rho_1^{\ell_1}\cdots \rho_{d_1}^{\ell_{d_1}}{a_1-p+\ell-1\choose a_1-p-1,\ell_1,\ldots,\ell_{d_1}}\Bigg]_+\Bigg)
   \end{align}
   for each nonzero $(p,q)\in\ZZ_{\ge0}^2$.
  \end{definition}
  \begin{remark}
   One might more naturally call the elements $x[a_1,a_2]$ ``frugal" since they spend as little as possible to be positive, however the originals were coined ``greedy" and we will adhere to this terminology.
  \end{remark}
  Clearly, the greedy element $x[a_1,a_2]$ is unique and indecomposable once it is well-defined, i.e. once it is known that only finitely many of the coefficients $c(p,q)$ are nonzero.  To establish this we will give a combinatorial construction in the next section.  This combinatorial expression will also imply a more pleasant recursion which eliminates the need for the maximum.
  \begin{proposition}
   Fix $(a_1,a_2)\in\ZZ^2$ and define $c(p,q)$ by the recursion \eqref{eq:greedy recursion max}.  Then for each nonzero $(p,q)\in\ZZ^2_{\ge0}$ the expansion coefficient $c(p,q)$ admits the following closed form:
   \begin{align}
   \label{eq:greedy recursion}
    c(p,q)=
    &\begin{cases}
     \bigg[\sum\limits_{k=1}^p\sum\limits_{(k_1,\ldots,k_{d_2})\partition k}\!\!\!\!\!\!\! (-1)^{k-1} c(p-k_1-2k_2-\cdots-d_2k_{d_2},q) \varrho_1^{k_1}\cdots \varrho_{d_2}^{k_{d_2}}{a_2-q+k-1\choose a_2-q-1,k_1,\ldots,k_{d_2}}\bigg]_+ \!\!\!\!\!& \text{if $a_1 q\le a_2 p$;}\\
    \bigg[\sum\limits_{\ell=1}^q\sum\limits_{(\ell_1,\ldots,\ell_{d_1})\partition \ell}\!\!\!\!\! (-1)^{\ell-1} c(p,q-\ell_1-2\ell_2-\cdots-d_1\ell_{d_1}) \rho_1^{\ell_1}\cdots \rho_{d_1}^{\ell_{d_1}}{a_1-p+\ell-1\choose a_1-p-1,\ell_1,\ldots,\ell_{d_1}}\bigg]_+ & \text{if $a_1 q\ge a_2 p$.}\end{cases}
   \end{align}
  \end{proposition}
  
  \bigskip
  
  \subsection{Combinatorial Construction of Greedy Elements}
  The results of this section rely on the notation and results presented in Section~\ref{sec:compatible pairs} which may be read and understood independently from the results presented here and above.
  
  Fix $(a_1,a_2)\in\ZZ^2$.  Let $D=D_{[a_1]_+,[a_2]_+}$ denote a maximal Dyck path with horizontal edges $D_1$ and vertical edges $D_2$, see Section~\ref{sec:Dyck} for details.  Write $\cC=\cC_{[a_1]_+,[a_2]_+}$ for the set of compatible pairs in $D$, see Definition~\ref{def:compatibility}.  We begin by introducing the coefficients of the combinatorial construction of the greedy element $x[a_1,a_2]$.
  \begin{definition}
   For each compatible pair $(S_1,S_2)\in\cC$ we define the coefficient 
   \[c_{S_1,S_2}=\rho_1^{\ell_1}\cdots\rho_{d_1}^{\ell_{d_1}}\varrho_1^{k_1}\cdots\varrho_{d_2}^{k_{d_2}},\]
   where $\ell_r$ and $k_r$ are the cardinalities of the sets $\{h\in D_1\st S_1(h)=r\}$ and $\{v\in D_2\st S_2(v)=r\}$ respectively.  We will also write $c_{S_1}=\rho_1^{\ell_1}\cdots\rho_{d_1}^{\ell_{d_1}}$ and $c_{S_2}=\varrho_1^{k_1}\cdots\varrho_{d_2}^{k_{d_2}}$ with notation as above, so that $c_{S_1,S_2}=c_{S_1}c_{S_2}$.
  \end{definition}
  \begin{theorem}
  \label{th:combinatorial greedy}
   For any $(a_1,a_2)\in\ZZ^2$ the greedy element $x[a_1,a_2]$ can be computed via
   \begin{equation}\label{eq:combinatorial greedy}
    x[a_1,a_2]=x_1^{-a_1}x_2^{-a_2}\sum\limits_{(S_1,S_2)\in\cC}c_{S_1,S_2}x_1^{|S_2|}x_2^{|S_1|}.
   \end{equation}
  \end{theorem}
  \begin{example}
   We continue Example~\ref{ex:generalized variables} and will use the compatible pairs given in Example~\ref{ex:compatible pairs}.  For $P_1(z)=1+z+z^2$ and $P_2(z)=1+z+z^2+z^3$ we may easily match the cluster variable $x_5$ with the right hand side of \eqref{eq:combinatorial greedy} for $(a_1,a_2)=(5,2)$.  Indeed, one may associate conveniently factorized expressions to the compatible pairs presented in Example~\ref{ex:compatible pairs}, for example the following compatible pairs contribute to summands containing $x_1^4$:
   \begin{center}
   \hfill
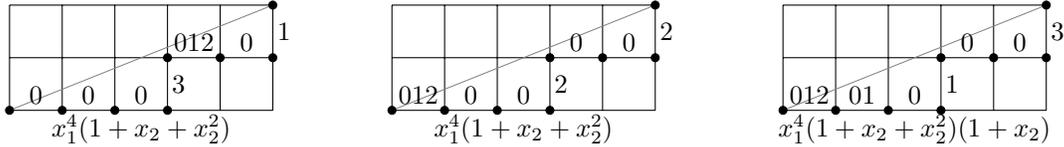
   \begin{tikzpicture}
    \draw[step=0.7cm,color=black] (0,0) grid (3.5,1.4);
    \draw[color=gray,very thin] (0,0) -- (3.5,1.4);
    \draw[fill=black] (0,0) circle (1.5pt);
    \draw (0.35,0.2) node {$0$};
    \draw[fill=black] (0.7,0) circle (1.5pt);
    \draw (1.05,0.2) node {$0$};
    \draw[fill=black] (1.4,0) circle (1.5pt);
    \draw (1.75,0.2) node {$0$};
    \draw[fill=black] (2.1,0) circle (1.5pt);
    \draw (2.25,0.35) node {$3$};
    \draw[fill=black] (2.1,0.7) circle (1.5pt);
    \draw (2.45,0.9) node {$012$};
    \draw[fill=black] (2.8,0.7) circle (1.5pt);
    \draw (3.15,0.9) node {$0$};
    \draw[fill=black] (3.5,0.7) circle (1.5pt);
    \draw (3.65,1.05) node {$1$};
    \draw[fill=black] (3.5,1.4) circle (1.5pt);
    \draw (1.75,-0.25) node {$x_1^4(1+x_2+x_2^2)$};
   \end{tikzpicture}
   \hfill
   \begin{tikzpicture}
    \draw[step=0.7cm,color=black] (0,0) grid (3.5,1.4);
    \draw[color=gray,very thin] (0,0) -- (3.5,1.4);
    \draw[fill=black] (0,0) circle (1.5pt);
    \draw (0.35,0.2) node {$012$};
    \draw[fill=black] (0.7,0) circle (1.5pt);
    \draw (1.05,0.2) node {$0$};
    \draw[fill=black] (1.4,0) circle (1.5pt);
    \draw (1.75,0.2) node {$0$};
    \draw[fill=black] (2.1,0) circle (1.5pt);
    \draw (2.25,0.35) node {$2$};
    \draw[fill=black] (2.1,0.7) circle (1.5pt);
    \draw (2.45,0.9) node {$0$};
    \draw[fill=black] (2.8,0.7) circle (1.5pt);
    \draw (3.15,0.9) node {$0$};
    \draw[fill=black] (3.5,0.7) circle (1.5pt);
    \draw (3.65,1.05) node {$2$};
    \draw[fill=black] (3.5,1.4) circle (1.5pt);
    \draw (1.75,-0.25) node {$x_1^4(1+x_2+x_2^2)$};
   \end{tikzpicture}
   \hfill
   \begin{tikzpicture}
    \draw[step=0.7cm,color=black] (0,0) grid (3.5,1.4);
    \draw[color=gray,very thin] (0,0) -- (3.5,1.4);
    \draw[fill=black] (0,0) circle (1.5pt);
    \draw (0.35,0.2) node {$012$};
    \draw[fill=black] (0.7,0) circle (1.5pt);
    \draw (1.05,0.2) node {$01$};
    \draw[fill=black] (1.4,0) circle (1.5pt);
    \draw (1.75,0.2) node {$0$};
    \draw[fill=black] (2.1,0) circle (1.5pt);
    \draw (2.25,0.35) node {$1$};
    \draw[fill=black] (2.1,0.7) circle (1.5pt);
    \draw (2.45,0.9) node {$0$};
    \draw[fill=black] (2.8,0.7) circle (1.5pt);
    \draw (3.15,0.9) node {$0$};
    \draw[fill=black] (3.5,0.7) circle (1.5pt);
    \draw (3.65,1.05) node {$3$};
    \draw[fill=black] (3.5,1.4) circle (1.5pt);
    \draw (1.75,-0.25) node {$x_1^4(1+x_2+x_2^2)(1+x_2)$};
   \end{tikzpicture},
   \hfill\hfill\hfill
   \end{center}
   which clearly agrees with the coefficient of $x_1^4$ in $x_5$.  The expressions associated to all other compatible pairs can be found in Example~\ref{ex:compatible pairs}.
  \end{example}
  
  To establish the combinatorial construction of greedy elements we will need some preliminary results, but first a definition.  Using the symmetry of the exchange relations \eqref{eq:exchange relation} we see that the generalized cluster algebra $\cA(P_1,P_2)$ admits reflection automorphisms $\sigma_p$ ($p\in\ZZ$) defined by $\sigma_p(x_k)=x_{2p-k}$.  One easily checks that the reflections $\{\sigma_p\}_{p\in\ZZ}$ generate a (possibly infinite) dihedral group and that this group may be generated by $\sigma_1$ and $\sigma_2$ alone.  In the next result we take \eqref{eq:combinatorial greedy} as the definition of the greedy element $x[a_1,a_2]$.
  
  \begin{proposition}
  \label{prop:combinatorial greedy symmetry}
   For any $(a_1,a_2)\in\ZZ^2$ the elements $x[a_1,a_2]$ defined by \eqref{eq:combinatorial greedy} satisfy the following symmetry properties:
   \begin{equation}\label{eq:reflection greedy}
    \sigma_1(x[a_1,a_2])=x[a_1,d_1[a_1]_+-a_2]\quad\text{and}\quad\sigma_2(x[a_1,a_2])=x[d_2[a_2]_+-a_1,a_2],
   \end{equation}
   where $[a]_+=\max(a,0)$.
  \end{proposition}
  
  \begin{proof}
   By symmetry it suffices to establish the second identity in \eqref{eq:reflection greedy}.  We will have a few cases to consider.\\
   
    \noindent $\bullet$ Suppose $a_1,a_2\le0$.  Then $D_{[a_1]_+,[a_2]_+}=D_{0,0}$ consists of a single point and we have
    \[\sigma_2(x[a_1,a_2])=\sigma_2(x_1^{-a_1}x_2^{-a_2})=x_3^{-a_1}x_2^{-a_2}.\]
    Using the identity $x_3=\frac{P_1(x_2)}{x_1}$ this becomes $x_1^{a_1}P_1(x_2)^{-a_1}x_2^{-a_2}$.  But notice that the maximal Dyck path $D_{-a_1,0}$ consists of exactly $-a_1$ consecutive horizontal edges and no vertical edges.  In particular, in every compatible pair $(S_1,S_2)\in\cC_{-a_1,0}$ the vertical grading $S_2$ is trivial and every possible horizontal grading $S_1:D_1\to[0,d_1]$ occurs in $\cC_{-a_1,0}$.  Thus we see that $x[-a_1,a_2]=x_1^{a_1}x_2^{-a_2}P_1(x_2)^{-a_1}$, as claimed.\\
    
    \noindent $\bullet$ Suppose $a_1\le0<a_2$.  Then $D_{[a_1]_+,[a_2]_+}$ consists of exactly $a_2$ consecutive vertical edges and no horizontal edges.  As in the previous case it is easy to see that $x[a_1,a_2]=x_1^{-a_1}x_2^{-a_2}P_2(x_1)^{a_2}$.  Applying $\sigma_2$ and the identity $x_3=\frac{P_1(x_2)}{x_1}$ gives
    \begin{align*}
     \sigma_2(x[a_1,a_2])
     &=x_3^{-a_1}x_2^{-a_2}P_2(x_3)^{a_2}\\
     &=x_1^{a_1}x_2^{-a_2}P_1(x_2)^{-a_1}P_2\big(P_1(x_2)x_1^{-1}\big)^{a_2}\\
     &=x_1^{-d_2a_2+a_1}x_2^{-a_2}P_1(x_2)^{-a_1}\big(x_1^{d_2}+\varrho_{d_2-1}x_1^{d_2-1}P_1(x_2)+\cdots+\varrho_1x_1P_1(x_2)^{d_2-1}+P_1(x_2)^{d_2}\big)^{a_2},
    \end{align*}
    where we used that $P_2$ is palindromic to reverse the coefficients.  Thus we need to show that 
    \[P_1(x_2)^{-a_1}\big(x_1^{d_2}+\varrho_{d_2-1}x_1^{d_2-1}P_1(x_2)+\cdots+\varrho_1x_1P_1(x_2)^{d_2-1}+P_1(x_2)^{d_2}\big)^{a_2}=\sum\limits_{(S_1,S_2)\in\cC_{d_2a_2-a_1,a_2}}c_{S_1,S_2}x_1^{|S_2|}x_2^{|S_1|}.\]
    Fix a partition $(k_0,k_1,\ldots,k_{d_2})\partition a_2$.  On the left we will expand using \eqref{eq:multinomial definition} and take the coefficient of $x_1^{k_1+2k_2+\cdots+d_2k_{d_2}}$ in 
    \[{a_2\choose k_0,k_1,\ldots,k_{d_2}}\big(P_1(x_2)^{d_2}\big)^{k_0}\big(\varrho_1x_1P_1(x_2)^{d_2-1}\big)^{k_1}\cdots\big(\varrho_{d_2-1}x_1^{d_2-1}P_1(x_2)\big)^{k_{d_2-1}}\big(x_1^{d_2}\big)^{k_{d_2}}.\]
    On the right we consider the sum $\sum\limits_{(S_1,S_2)}c_{S_1,S_2}x_1^{|S_2|}x_2^{|S_1|}$ over all compatible pairs such that $k_r$ is the cardinality of the set $\{v\in D_2\st S_2(v)=r\}$, notice that there are exactly ${a_2\choose k_0,k_1,\ldots,k_{d_2}}$ such vertical gradings $S_2$ and we will always have $|S_2|=k_1+2k_2+\cdots+d_2k_{d_2}$.  Thus to complete the proof it suffices to show
    \[\sum\limits_{S_1\in\cC(S_2)}c_{S_1,S_2}x_2^{|S_1|}=P_1(x_2)^{-a_1}\big(P_1(x_2)^{d_2}\big)^{k_0}\big(\varrho_1P_1(x_2)^{d_2-1}\big)^{k_1}\cdots\big(\varrho_{d_2-1}P_1(x_2)\big)^{k_{d_2-1}}\]
    for each vertical grading $S_2$ as above, where the summation runs over all horizontal gradings $S_1$ such that $(S_1,S_2)\in\cC_{d_2a_2-a_1,a_2}$.  Using that $c_{S_1,S_2}=c_{S_1}c_{S_2}$ and that in this case we have $c_{S_2}=\varrho_1^{k_1}\cdots\varrho_{d_2-1}^{k_{d_2-1}}$, we may further reduce the problem to showing that
    \begin{equation}\label{eq:outside shadow weights}
     \sum\limits_{S_1\in\cC(S_2)}c_{S_1}x_2^{|S_1|}=P_1(x_2)^{-a_1}P_1(x_2)^{d_2k_0}P_1(x_2)^{(d_2-1)k_1}\cdots P_1(x_2)^{k_{d_2-1}}.
    \end{equation}
    To see this notice that, by Lemma~\ref{le:shadows}, any horizontal edge $h$ outside $\sh(S_2)$ may be assigned any weight $S_1(h)$ without affecting compatibility.  At most $S_2(v)=d_2$ for every vertical edge $v$, which leaves $-a_1$ horizontal edges outside the shadow of $S_2$, this accounts for the factor $P_1(x_2)^{-a_1}$.  Finally for $0\le r\le d_2$ consider the $k_r$ vertical edges $v$ such that $S_2(v)=r$.  There will be $d_2-r$ horizontal edges outside $\sh(S_2)$ for each such vertical edge, this contributes the factor $P_1(x_2)^{(d_2-r)k_r}$.  This completes the proof of \eqref{eq:outside shadow weights} and thus the proof of the proposition in this case.\\
    
    \noindent $\bullet$ Suppose $0<a_1,a_2$.  Then by definition $x[a_1,a_2]=x_1^{-a_1}x_2^{-a_2}\sum\limits_{(S_1,S_2)\in\cC_{a_1,a_2}}c_{S_1,S_2}x_1^{|S_2|}x_2^{|S_1|}$.  We apply $\sigma_2$ and the identity $x_3=\frac{P_1(x_2)}{x_1}$ to get
    \begin{align*}
     \sigma_2(x[a_1,a_2])
     &=x_3^{-a_1}x_2^{-a_2}\sum\limits_{(S_1,S_2)\in\cC_{a_1,a_2}}c_{S_1,S_2}x_3^{|S_2|}x_2^{|S_1|}\\
     &=x_1^{-d_2a_2+a_1}x_2^{-a_2}P_1(x_2)^{-a_1}\sum\limits_{(S_1,S_2)\in\cC_{a_1,a_2}}c_{S_1,S_2}x_1^{d_2a_2-|S_2|}P_1(x_2)^{|S_2|}x_2^{|S_1|}.
    \end{align*}
    In the notation of Section~\ref{sec:bounded} notice that $|\varphi_{d_2}^*S_2|=d_2a_2-|S_2|$.  Thus comparing with the definition \eqref{eq:combinatorial greedy} of $x[d_2a_2-a_1,a_2]$ we see that the claim will follow if we can establish the following identity:
    \begin{equation}\label{eq:weight equality}
     \sum\limits_{S_1\in\cC(S_2)}c_{S_1,S_2}x_2^{|S_1|}=P_1(x_2)^{a_1-|S_2|}\sum\limits_{S'_1\in\cC(\varphi_{d_2}^*S_2)}c_{S'_1,\varphi_{d_2}^*S_2}x_2^{|S'_1|}
    \end{equation}
    for each vertical grading $S_2$.  According to Lemma~\ref{le:shadows} the number of horizontal edges $|D_1\setminus\sh(S_2)|$ outside the shadow of $S_2$ is
    \[a_1-\min(a_1,|S_2|)=[a_1-|S_2|]_+.\]
    Similarly the number of horizontal edges $|D_1'\setminus\sh(\varphi_{d_2}^*S_2)|$ outside the shadow of $\varphi_{d_2}^*S_2$ is
    \[d_2a_2-a_1-\min(d_2a_2-a_1,|\varphi_{d_2}^*S_2|)=[d_2a_2-a_1-|\varphi_{d_2}^*S_2|]_+=[|S_2|-a_1]_+=[a_1-|S_2|]_+-(a_1-|S_2|).\]
    Since each of these edges contributes a factor of $P_1(x_2)$ we can rewrite \eqref{eq:weight equality} as
    \[P_1(x_2)^{[a_1-|S_2|]_+}\sum\limits_{S_1\in\cC_{rs}(S_2)}c_{S_1,S_2}x_2^{|S_1|}=P_1(x_2)^{[a_1-|S_2|]_+}\sum\limits_{S'_1\in\cC_{rs}(\varphi_{d_2}^*S_2)}c_{S'_1,\varphi_{d_2}^*S_2}x_2^{|S'_1|},\]
    where $\cC_{rs}(S_2)$ denotes those horizontal gradings $S_1\in\cC(S_2)$ such that $\supp(S_1)\subset\rsh(S_2)$.  
    
    Since $P_2$ is palindromic we have $c_{S_2}=c_{\varphi_{d_2}^*S_2}$.  From the definition \eqref{eq:omega definition} of $\Omega$ we have $c_{S_1}=c_{\Omega(S_1)}$ and $|S_1|=|\Omega(S_1)|$ for each $S_1\in\cC_{rs}(S_2)$.  The result then follows by applying the bijection $\Omega:\cC_{rs}(S_2)\to\cC_{rs}(\varphi_{d_2}^*S_2)$ established in Proposition~\ref{prop:horizontal grading bijection}.
  \end{proof}
  
  The next result says that the combinatorially defined greedy elements are actually elements of the generalized cluster algebra.
  \begin{corollary}
   For any $a_1,a_2\in\ZZ$ the elements $x[a_1,a_2]\in\cT_1$ defined by \eqref{eq:combinatorial greedy} are contained in $\cA(P_1,P_2)$.
  \end{corollary}
  \begin{proof}
   By Proposition~\ref{prop:combinatorial greedy symmetry} we have $x[a_1,a_2]=\sigma_1(x[a_1,d_1[a_1]_+-a_2])\in\cT_0$ and $x[a_1,a_2]=\sigma_2(x[d_2[a_2]_+-a_1,a_2])\in\cT_2$.  The result is then a consequence of Theorem~\ref{th:strong Laurent}.
  \end{proof}
  We conclude this section by showing that the combinatorially defined elements $x[a_1,a_2]$ satisfy the recursion \eqref{eq:greedy recursion max}.  More precisely, we will show that they satisfy \eqref{eq:greedy recursion} and then appealing to Proposition~\ref{prop:positive elements} we may conclude that they satisfy \eqref{eq:greedy recursion max}.
  
  \begin{proposition}
  \label{prop:combinatorial construction}
   For any $a_1,a_2\ge0$ the elements $x[a_1,a_2]\in\cT_1$ defined by \eqref{eq:combinatorial greedy} satisfy \eqref{eq:greedy recursion}.
  \end{proposition}
  \begin{proof}
   By symmetry it suffices to show the second equality, that is
   \begin{equation}
   \label{eq:coefficient recursion}
   c(p,q)=\Bigg[\sum\limits_{\ell=1}^q\sum\limits_{(\ell_1,\ldots,\ell_{d_1})\partition \ell}\!\!\!\!\! (-1)^{\ell-1} c(p,q-\ell_1-2\ell_2-\cdots-d_1\ell_{d_1}) \rho_1^{\ell_1}\cdots \rho_{d_1}^{\ell_{d_1}}{a_1-p+\ell-1\choose a_1-p-1,\ell_1,\ldots,\ell_{d_1}}\Bigg]_+
   \end{equation}
   for $a_1q\ge a_2p$ and $(p,q)\ne(0,0)$.  The remainder of the proof splits into two cases.\\
   
   \noindent $\bullet$ Suppose $a_1\le p$.  First notice that $(-1)^{\ell-1}{a_1-p+\ell-1\choose a_1-p-1,\ell_1,\ldots,\ell_{d_1}}$ is always nonpositive and thus the right hand side is zero.  Now the inequality $a_1\le p$ together with $a_1q\ge a_2p$ implies $q\ge a_2$.  But then Lemma~\ref{le:grading bound} implies $c(p,q)=0$ since we may identify $p=|S_2|$ and $q=|S_1|$ for some compatible pair $(S_1,S_2)\in\cC_{a_1,a_2}$.\\
   
   \noindent $\bullet$ Suppose $p<a_1$.  Then as in \eqref{eq:expansion coefficient} we may interpret the quantity
   \begin{equation}
   \label{eq:coefficient}
    \sum\limits_{\ell\ge0}\sum\limits_{(\ell_1,\ldots,\ell_{d_1})\partition \ell} (-1)^\ell c(p,q-\ell_1-2\ell_2-\cdots-d_1\ell_{d_1}) \rho_1^{\ell_1}\cdots \rho_{d_1}^{\ell_{d_1}}{a_1-p+\ell-1\choose a_1-p-1,\ell_1,\ldots,\ell_{d_1}}
   \end{equation}
   as the the coefficient of $x_2^{q-a_2}x_3^{a_1-p}$ in the expansion of $x[a_1,a_2]$ as an element of $\cT_2$.  Applying $\sigma_2$ we may interpret this as the coefficient of $x_1^{a_1-p}x_2^{q-a_2}$ in $\sigma_2(x[a_1,a_2])=x[d_2a_2-a_1,a_2]$.  If we denote the expansion coefficients of $x[d_2a_2-a_1,a_2]$ by $c'(p',q')$ then we may identify the quantity \eqref{eq:coefficient} with $c'(d_2a_2-p,q)$.  We will show that our current assumptions imply this coefficient is zero which implies \eqref{eq:coefficient recursion}.  
   
   To simplify the notation we will abbreviate $a'_1=d_2a_2-a_1$ and $p'=d_2a_2-p$.  Denote compatible pairs in $D':=D_{[a'_1]_+,a_2}$ by $(S'_1,S'_2)\in\cC_{[a'_1]_+,a_2}$.  Then note that $c'(p',q)$ is nonzero if and only if there exists a compatible pair with $|S'_1|=q$ and $|S'_2|=p'>a'_1$.  Moreover, since we assume $(p,q)\ne(0,0)$ we have $(|S'_1|,|S'_2|)\ne(0,d_2a_2)$.  Finally, the assumption $a_1q\ge a_2p$ translates into the inequality 
   \begin{equation}
   \label{eq:contradicting inequality}
    (d_2a_2-a'_1)|S'_1|\ge a_2(d_2a_2-|S'_2|).
   \end{equation}
   We now match with the results from Proposition~\ref{prop:supports} by splitting into further subcases.  First note that $d_2a_2\le a'_1$ is impossible since we assume $a_1>0$ in this case.\\
   
   \noindent --- Suppose $a'_1\le0$.  Then $D'$ consists of $a_2$ consecutive vertical edges and no horizontal edges so that $(|S'_1|,|S'_2|)$ is contained in the segment $\big[(0,0),(0,d_2a_2)\big)$.  Since $|S'_1|=0$, the inequality \eqref{eq:contradicting inequality} becomes $d_2a_2\le|S'_2|$ which is impossible.\\
   
   \noindent --- Suppose $0<d_1a'_1\le a_2$ and $0<a'_1<d_2a_2$.  Then by Proposition~\ref{prop:supports}.b the point $(|S'_1|,|S'_2|)$ must lie on or below the segment $\big[(0,d_2a_2),(d_1a'_1,d_2a_2-d_1d_2a'_1)\big]$, which is equivalent to the inequality 
   \[|S'_2|\le -d_2|S'_1|+d_2a_2.\]
   Combining this with \eqref{eq:contradicting inequality} gives the inequalities $0\ge|S'_1|$ and $|S'_2|\ge d_2a_2$, which are only satisfied for $(|S'_1|,|S'_2|)=(0,d_2a_2)$.  However we have assumed this is not the case, thus there are no compatible pairs with $(|S'_1|,|S'_2|)=(q,p')$ and \eqref{eq:coefficient} is equal to zero.\\
   
   \noindent --- Suppose $0<a_2<d_1a'_1$ and $0<a'_1<d_2a_2$.  Then Proposition~\ref{prop:supports}.c says $(|S'_1|,|S'_2|)$ must lie strictly below the segment $\big[(0,d_2a_2),(a_2,a'_1)\big]$ (actually Proposition~\ref{prop:supports} allows $(|S'_1|,|S'_2|)=(0,d_2a_2)$ but we have excluded this point), which is equivalent to the inequality
   \[|S'_2|<-\frac{d_2a_2-a'_1}{a_2}|S'_1|+d_2a_2.\]
   But notice that this is exactly the complementary inequality to \eqref{eq:contradicting inequality}, i.e. 
   \[(d_2a_2-a'_1)|S'_1|< a_2(d_2a_2-|S'_2|),\]
   thus there are no compatible pairs with $(|S'_1|,|S'_2|)=(q,p')$ and \eqref{eq:coefficient} is equal to zero.\\
   
   This completes the proof of Proposition~\ref{prop:combinatorial construction}.
  \end{proof}
  
  \begin{proof}[Proof of Theorem~\ref{th:combinatorial greedy}] 
   Consider $(a_1,a_2)\in\ZZ^2$.  Using Proposition~\ref{prop:positive elements} and Proposition~\ref{prop:combinatorial construction} we may conclude for $(a_1,a_2)\in\ZZ_{\ge0}^2$ that the combinatorial element defined by \eqref{eq:combinatorial greedy} satisfies \eqref{eq:greedy recursion max} and thus that it is actually the greedy element $x[a_1,a_2]$.  For $(a_1,a_2)\in\ZZ_{<0}^2$ both terms in \eqref{eq:greedy recursion max} are always zero and $D_{[a_1]_+,[a_2]_+}=D_{0,0}$, thus the combinatorially defined element and the greedy element $x[a_1,a_2]$ are both equal to the initial cluster monomial $x_1^{-a_1}x_2^{-a_2}$.
   
   Now suppose $a_1>0$ and $a_2<0$.  Notice that the first term in \eqref{eq:greedy recursion max} is always zero for $a_2\le0$.  Using the second term in \eqref{eq:greedy recursion max} we see that $c(p,0)=0$ for all $p>0$ and thus $c(p,q)=0$ whenever $p>0$.  Comparing with \eqref{eq:x pointed expansion} we may conclude that $x[a_1,a_2]=x_2^{-a_2}x[a_1,0]$, where we note that $x[a_1,0]$ has already identified with the combinatorial element defined by \eqref{eq:combinatorial greedy}.  Using Proposition~\ref{prop:combinatorial greedy symmetry} we may conclude that $x[a_1,0]=x_3^{a_1}$ and thus $x[a_1,a_2]=x_2^{-a_2}x_3^{a_1}=\sigma_2(x[-a_1,a_2])$ coincides with the combinatorially defined element.
   
   The case $a_1<0$ and $a_2>0$ follows by a similar argument.
  \end{proof}
  
  We are now ready to complete the proof of our main theorem.
  \begin{proof}[Proof of Theorem~\ref{th:main}]
   The existence of greedy elements is the content of Theorem~\ref{th:combinatorial greedy}, uniqueness and indecomposability follow from the definition \eqref{eq:greedy recursion max}.  This gives parts (a) and (b).  
   
   In the proof of Theorem~\ref{th:combinatorial greedy} above we have seen that $x[a_1,a_2]=z_1[a_1,a_2]$ is a standard monomial for $(a_1,a_2)\in\ZZ^2\setminus\ZZ_{>0}^2$, i.e. the set $\big\{x[a_1,a_2]\big\}_{a_1,a_2\in\ZZ}$ of greedy elements is a complete bounded collection of pointed elements in $\cA(P_1,P_2)$, thus part (c) follows from Proposition~\ref{prop:bounded basis}.
   
   To see parts (d) and (e) we note that $x[a_1,a_2]=x_1^{-a_1}x_2^{-a_2}$ is an initial cluster monomial for every $(a_1,a_2)\in\ZZ_{<0}^2$.  Since any other cluster monomial can be obtained from an initial cluster monomial by a reflection $\sigma_p$ for some $p\in\ZZ$ (or better, by alternately applying $\sigma_1$ and $\sigma_2$), we see that both part (d) and (e) follow from Proposition~\ref{prop:combinatorial greedy symmetry}.
  \end{proof}

  To complete this section we describe explicitly which greedy elements correspond to cluster monomials.  To do so we introduce two-parameter Chebyshev polynomials $u_{k,j}$ ($k,j\in \ZZ$) defined recursively by: 
  \[u_{-1,j}=0, \quad u_{0,j}=1,\quad u_{k+1,j+1}=d_ju_{k,j}-u_{k-1,j-1} \text{ where $d_j=\begin{cases} d_1, & \text{ if $j$ is odd;}\\ d_2, & \text{ if $j$ is even.}\end{cases}$}\]
  \begin{remark}
   The two-parameter Chebyshev polynomials can be seen as the components of positive roots in the rank two root system associated to the Cartan matrix $\left(\begin{array}{cc}2 & -d_1\\ -d_2 & 2\end{array}\right)$.  Since we will not need this fact we leave the verification to the reader.
 \end{remark}
   Using Proposition~\ref{prop:combinatorial greedy symmetry} and the definition of the Chebyshev polynomials we easily see that the cluster variables are given by
   \begin{equation}
   \label{eq:greedy cluster variable}
    x_k=\begin{cases}x[u_{k-3,1},u_{k-4,2}] & \text{ for $k\ge2$;}\\ x[u_{-k-1,1},u_{-k,2}] & \text{ for $k\le1$.}\end{cases}
   \end{equation}
   Applying a sequence of reflections $\sigma_1$ and $\sigma_2$ we can transform 
   \[(x_1,x_2)\mapsto\begin{cases}(x_k,x_{k+1}) & \text{ for odd $k$;}\\ (x_{k+1},x_k) & \text{ for even $k$.}\end{cases}\]
   Applying the same sequence of reflections to the initial cluster monomial $x[a_1,a_2]=x_1^{-a_1}x_2^{-a_2}$ for $(a_1,a_2)\in\ZZ_{\le0}^2$ and applying \eqref{eq:greedy cluster variable} we see that the $k^{th}$ cluster monomial $z_k[a_1,a_2]$ is given by
   \begin{equation}
   \label{eq:greedy cluster monomial}
    z_k[a_1,a_2]=x_k^{-a_1}x_{k+1}^{-a_2}=\begin{cases}x[-a_1u_{k-3,1}-a_2u_{k-2,1},-a_1u_{k-4,2}-a_2u_{k-3,2}] & \text{ for $k\ge2$;}\\ x[a_1,a_2] & \text{ for $k=1$;}\\ x[-a_1u_{-k-1,1}-a_2u_{-k-2,1},-a_1u_{-k,2}-a_2u_{-k-1,2}] & \text{ for $k\le0$.}\end{cases}
   \end{equation}
   Combining \eqref{eq:greedy cluster variable} and \eqref{eq:greedy cluster monomial} we get the following factorization properties for greedy elements corresponding to cluster monomials:
   \begin{equation}
    x[a_1u_{k-3,1}+a_2u_{k-2,1},a_1u_{k-4,2}+a_2u_{k-3,2}]=x[u_{k-3,1},u_{k-4,2}]^{a_1}x[u_{k-2,1},u_{k-3,2}]^{a_2}
   \end{equation}
   for $(a_1,a_2)\in\ZZ_{\ge0}^2$ and $k\ge2$;
   \begin{equation}
    x[a_1u_{-k-1,1}+a_2u_{-k-2,1},a_1u_{-k,2}+a_2u_{-k-1,2}]=x[u_{-k-1,1},u_{-k,2}]^{a_1}x[u_{-k-2,1},u_{-k-1,2}]^{a_2}
   \end{equation}
   for $(a_1,a_2)\in\ZZ_{\ge0}^2$ and $k\le0$.  
   
   \begin{remark}
    Consider any Dyck path $D$ of the form
    \[D_{a_1u_{k-3,1}+a_2u_{k-2,1},a_1u_{k-4,2}+a_2u_{k-3,2}}\quad\text{ or }\quad D_{a_1u_{-k-1,1}+a_2u_{-k-2,1},a_1u_{-k,2}+a_2u_{-k-1,2}}\]
    giving rise to a cluster monomial.  Using \eqref{eq:combinatorial greedy}, the factorization equations above imply that the ``wrap around" of subpaths in $D$ assumed in Section~\ref{sec:Dyck} can be dropped when considering compatible pairs.  In other words, there is a principle of non-interaction between the various subpaths of $D$ of the form
    \[D_{u_{k-3,1},u_{k-4,2}}\quad\text{ and }\quad D_{u_{k-2,1},u_{k-3,2}}\]
    or
    \[D_{u_{-k-1,1},u_{-k,2}}\quad\text{ and }\quad D_{u_{-k-2,1},u_{-k-1,2}},\]
    i.e. when considering compatible pairs we may view compatible pairs on each such subpath independent of all other such subpaths of $D$.\\
   \end{remark}

\section{Structure of Maximal Dyck Paths}\label{sec:Dyck}

 For non-negative integers $a_1,a_2\in\ZZ_{\ge0}$, let $R=R_{a_1,a_2}$ denote the rectangle in $\RR^2$ with corner vertices $(0,0)$ and $(a_1,a_2)$.  A Dyck path in $R$ is a lattice path in $\ZZ^2\subset\RR^2$ beginning at $(0,0)$, taking North and East steps, and ending at $(a_1,a_2)$, where the path never crosses above the main diagonal of $R$.  We will consider the maximal Dyck path $D=D_{a_1,a_2}$ which lies closest to (and possibly touches) the main diagonal of $R$, i.e. any lattice point above $D$ is also above the main diagonal.
 
 Label the horizontal and vertical edges of $D$ as $D_1=\{h_1,\ldots,h_{a_1}\}$ and $D_2=\{v_1,\ldots,v_{a_2}\}$ respectively, where edges are labeled by their maximum distance from the vertical and horizontal axes respectively.  We will call the distance from a horizontal edge to the horizontal axis its \emph{height} and the distance from a vertical edge to the vertical axis its \emph{depth}.  It will be convenient to fold the rectangle $R$ into a torus and identify the vertices $(0,0)$ and $(a_1,a_2)$, in this way $D$ becomes a closed loop.  Thus we will sometimes allow arbitrary integer subscripts $h_j$ and $v_j$ with the understanding that $h_j=h_{j+a_1}$ and $v_j=v_{j+a_2}$ for $j\in\ZZ$.

 For edges $e,e'\in D$ we will let $ee'$ denote the subpath of $D$ beginning at $e$ and ending at $e'$.  We consider $ee$ to be the path which only contains the edge $e$.  Note that when $e$ lies to the North-East of $e'$ the path $ee'$ will contain the vertex $(0,0)\equiv(a_1,a_2)$ and ``wrap around'' $D$.  We will write $(ee')_1$ for the set of horizontal edges in the subpath $ee'$ and $(ee')_2$ for its set of vertical edges.  It will also be convenient to denote by $\overline{e}e'$, $e\overline{e}'$, and $\overline{e}\overline{e}'$ the paths obtained from $ee'$ by removing the edge $e$, removing the edge $e'$, or removing both edges $e$ and $e'$ respectively.
 
 \begin{example}
 \label{ex:dyck}
  For $(a_1,a_2)=(5,2)$ the maximal Dyck path $D$ has horizontal edges $D_1=\{h_1,h_2,h_3,h_4,h_5\}$ and vertical edges $D_2=\{v_1,v_2\}$ which can be visualized as\\
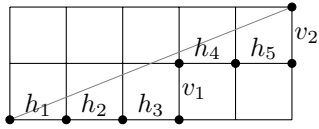
  \begin{center}
   \begin{tikzpicture}
    \draw[step=0.75cm,color=black] (0,0) grid (3.75,1.5);
    \draw[color=gray,very thin] (0,0) -- (3.75,1.5);
    \draw[fill=black] (0,0) circle (1.5pt);
    \draw (0.375,0.2) node {$h_1$};
    \draw[fill=black] (0.75,0) circle (1.5pt);
    \draw (1.125,0.2) node {$h_2$};
    \draw[fill=black] (1.5,0) circle (1.5pt);
    \draw (1.875,0.2) node {$h_3$};
    \draw[fill=black] (2.25,0) circle (1.5pt);
    \draw (2.45,0.375) node {$v_1$};
    \draw[fill=black] (2.25,0.75) circle (1.5pt);
    \draw (2.625,0.95) node {$h_4$};
    \draw[fill=black] (3,0.75) circle (1.5pt);
    \draw (3.375,0.95) node {$h_5$};
    \draw[fill=black] (3.75,0.75) circle (1.5pt);
    \draw (3.95,1.125) node {$v_2$};
    \draw[fill=black] (3.75,1.5) circle (1.5pt);
   \end{tikzpicture}.
  \end{center}
  Then $(h_3v_2)_1=\{h_3,h_4,h_5\}$ and $(h_3v_2)_2=\{v_1,v_2\}$, while $(v_2h_3)_1=\{h_1,h_2,h_3\}$ and $(v_2h_3)_2=\{v_2\}$.  The path $h_3h_3$ has length 1, while $v_1h_3=D$ has length 7.\\
 \end{example}

 We will often need the following easy but incredibly useful Lemma presented in \cite{lee-li-zelevinsky} which precisely describes the Dyck path $D$.\\
 
 \begin{lemma}\label{le:height_depth}\mbox{}
  \begin{enumeratea}
   \item The height of the horizontal edge $h_j$ is $\lfloor (j-1)a_2/a_1\rfloor$, and so the vertical distance $|(h_ih_j)_2|$ between $h_i$ and $h_j$ for $1\le i<j\le a_1$ is equal to 
  \[|(h_ih_j)_2|=\lfloor (j-1)a_2/a_1\rfloor-\lfloor (i-1)a_2/a_1\rfloor.\] 
   \item The depth of the vertical edge $v_j$ is $\lceil ja_1/a_2\rceil$, and so the horizontal distance $|(v_iv_j)_1|$ between $v_i$ and $v_j$ for $1\le i<j\le a_2$ is equal to 
  \[|(v_iv_j)_1|=\lceil ja_1/a_2\rceil-\lceil ia_1/a_2\rceil.\] 
  \end{enumeratea}
 \end{lemma}
 
 In particular, this implies the following result giving a bound on the slope of certain subpaths of $D$ in relation to the slope of the main diagonal of $R$.
 \begin{corollary}\label{cor:slopes}
  For any $h\in D_1$ and $v_j\in D_2$ with $h$ to the left/below $v_j$ we have $a_1\big(|(hv_j)_2|-1\big)<a_2|(hv_j)_1|$.
 \end{corollary}
 \begin{proof}
  Suppose $h$ has height $j'-1<j$.  Then $|(hv_j)_1|\ge|(v_{j'}v_j)_1|+1$ and $|(hv_j)_2|-1=j-j'$.  Then by Lemma~\ref{le:height_depth} we have
  \begin{align*}
   a_2|(hv_j)_1|\ge a_2\big(|(v_{j'}v_j)_1|+1\big)&=a_2\big(\lceil ja_1/a_2\rceil-\lceil j'a_1/a_2\rceil+1\big)\\
   &>a_2(ja_1/a_2-j'a_1/a_2)=a_1(j-j')=a_1\big(|(hv_j)_2|-1\big).
  \end{align*}
 \end{proof}

\section{Combinatorics of Graded Compatible Pairs: Shadows}\label{sec:compatible pairs}

 Here we introduce our main combinatorial object of study, namely graded compatible pairs.  We will begin with general results that do not require an upper bound on the gradings and restrict to the bounded case when it becomes necessary.  Most of the results in this section are generalizations of, or are inspired by, similar results from \cite[Section 3]{lee-li-zelevinsky}.  Let $D=D_{a_1,a_2}$ be any maximal Dyck path with horizontal edges $D_1$ and vertical edges $D_2$.  

\subsection{Unbounded Gradings}
 \begin{definition}\label{def:compatibility}
  Consider functions $S_1:D_1\to\ZZ_{\ge0}$ and $S_2:D_2\to\ZZ_{\ge0}$ which we call \emph{horizontal} (resp. \emph{vertical}) \emph{gradings}.  We call the pair $(S_1,S_2)$ \emph{(graded) compatible} if for every $h\in D_1$ and $v\in D_2$ there exists an edge $e\in D$ so that at least one of the following conditions on the gradings is satisfied:
  \begin{align}
   \label{eq:compatibility1}\tag{HGC} &\text{$he$ is a proper subpath of $hv$}\quad\text{ and } \quad|(he)_2|=\sum\limits_{h'\in(he)_1}S_1(h');\\
   \label{eq:compatibility2}\tag{VGC} &\text{$ev$ is a proper subpath of $hv$}\quad\text{ and } \quad|(ev)_1|=\sum\limits_{v'\in(ev)_2}S_2(v').
  \end{align}
  Write $\cC=\cC_{a_1,a_2}$ for the set of all such compatible pairs $(S_1,S_2)$.
 \end{definition}
 \begin{remark}
 \label{rem:trivial_compatibility}
  For $h\in D_1$ with $S_1(h)=0$ the horizontal grading condition~\eqref{eq:compatibility1} of Definition~\ref{def:compatibility} is trivially satisfied for all $v\in D_2$ by taking $e=h$.  Similarly for $v\in D_2$ with $S_2(v)=0$ the vertical grading condition~\eqref{eq:compatibility2} is trivially satisfied.
 \end{remark} 
 
 \begin{example}
 \label{ex:compatible pairs}
  We continue Example~\ref{ex:dyck}.  Here we will consider bounded gradings $S_1\st D_1\to\{0,1,2\}$ and $S_2\st D_2\to\{0,1,2,3\}$ on the Dyck path $D$.  All compatible pairs are given below, we draw all horizontal gradings compatible with a given vertical grading where for example ``$\,012$" written over an edge $h$ means $S_1(h)$ can be either $0$, $1$, or $2$ and the pair $(S_1,S_2)$ will be compatible:\\
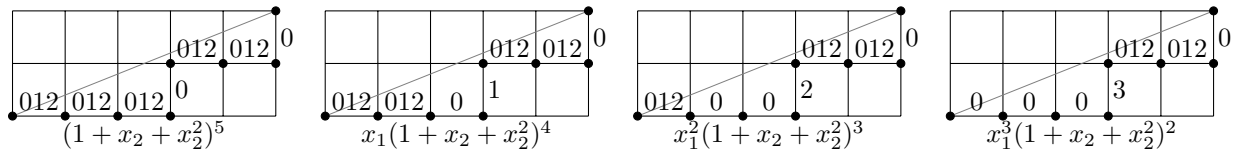
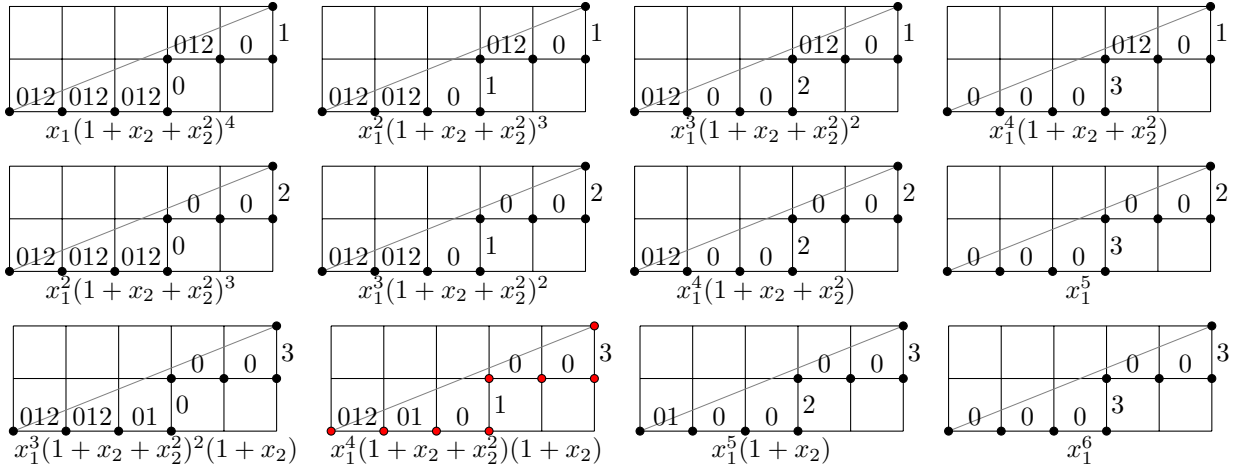
  \begin{center}
   \begin{tikzpicture}
    \draw[step=0.7cm,color=black] (0,0) grid (3.5,1.4);
    \draw[color=gray,very thin] (0,0) -- (3.5,1.4);
    \draw[fill=black] (0,0) circle (1.5pt);
    \draw (0.35,0.2) node {$012$};
    \draw[fill=black] (0.7,0) circle (1.5pt);
    \draw (1.05,0.2) node {$012$};
    \draw[fill=black] (1.4,0) circle (1.5pt);
    \draw (1.75,0.2) node {$012$};
    \draw[fill=black] (2.1,0) circle (1.5pt);
    \draw (2.25,0.35) node {$0$};
    \draw[fill=black] (2.1,0.7) circle (1.5pt);
    \draw (2.45,0.9) node {$012$};
    \draw[fill=black] (2.8,0.7) circle (1.5pt);
    \draw (3.15,0.9) node {$012$};
    \draw[fill=black] (3.5,0.7) circle (1.5pt);
    \draw (3.65,1.05) node {$0$};
    \draw[fill=black] (3.5,1.4) circle (1.5pt);
    \draw (1.75,-0.25) node {$(1+x_2+x_2^2)^5$};
   \end{tikzpicture}
   \hfill
   \begin{tikzpicture}
    \draw[step=0.7cm,color=black] (0,0) grid (3.5,1.4);
    \draw[color=gray,very thin] (0,0) -- (3.5,1.4);
    \draw[fill=black] (0,0) circle (1.5pt);
    \draw (0.35,0.2) node {$012$};
    \draw[fill=black] (0.7,0) circle (1.5pt);
    \draw (1.05,0.2) node {$012$};
    \draw[fill=black] (1.4,0) circle (1.5pt);
    \draw (1.75,0.2) node {$0$};
    \draw[fill=black] (2.1,0) circle (1.5pt);
    \draw (2.25,0.35) node {$1$};
    \draw[fill=black] (2.1,0.7) circle (1.5pt);
    \draw (2.45,0.9) node {$012$};
    \draw[fill=black] (2.8,0.7) circle (1.5pt);
    \draw (3.15,0.9) node {$012$};
    \draw[fill=black] (3.5,0.7) circle (1.5pt);
    \draw (3.65,1.05) node {$0$};
    \draw[fill=black] (3.5,1.4) circle (1.5pt);
    \draw (1.75,-0.25) node {$x_1(1+x_2+x_2^2)^4$};
   \end{tikzpicture}
   \hfill
   \begin{tikzpicture}
    \draw[step=0.7cm,color=black] (0,0) grid (3.5,1.4);
    \draw[color=gray,very thin] (0,0) -- (3.5,1.4);
    \draw[fill=black] (0,0) circle (1.5pt);
    \draw (0.35,0.2) node {$012$};
    \draw[fill=black] (0.7,0) circle (1.5pt);
    \draw (1.05,0.2) node {$0$};
    \draw[fill=black] (1.4,0) circle (1.5pt);
    \draw (1.75,0.2) node {$0$};
    \draw[fill=black] (2.1,0) circle (1.5pt);
    \draw (2.25,0.35) node {$2$};
    \draw[fill=black] (2.1,0.7) circle (1.5pt);
    \draw (2.45,0.9) node {$012$};
    \draw[fill=black] (2.8,0.7) circle (1.5pt);
    \draw (3.15,0.9) node {$012$};
    \draw[fill=black] (3.5,0.7) circle (1.5pt);
    \draw (3.65,1.05) node {$0$};
    \draw[fill=black] (3.5,1.4) circle (1.5pt);
    \draw (1.75,-0.25) node {$x_1^2(1+x_2+x_2^2)^3$};
   \end{tikzpicture}
   \hfill
   \begin{tikzpicture}
    \draw[step=0.7cm,color=black] (0,0) grid (3.5,1.4);
    \draw[color=gray,very thin] (0,0) -- (3.5,1.4);
    \draw[fill=black] (0,0) circle (1.5pt);
    \draw (0.35,0.2) node {$0$};
    \draw[fill=black] (0.7,0) circle (1.5pt);
    \draw (1.05,0.2) node {$0$};
    \draw[fill=black] (1.4,0) circle (1.5pt);
    \draw (1.75,0.2) node {$0$};
    \draw[fill=black] (2.1,0) circle (1.5pt);
    \draw (2.25,0.35) node {$3$};
    \draw[fill=black] (2.1,0.7) circle (1.5pt);
    \draw (2.45,0.9) node {$012$};
    \draw[fill=black] (2.8,0.7) circle (1.5pt);
    \draw (3.15,0.9) node {$012$};
    \draw[fill=black] (3.5,0.7) circle (1.5pt);
    \draw (3.65,1.05) node {$0$};
    \draw[fill=black] (3.5,1.4) circle (1.5pt);
    \draw (1.75,-0.25) node {$x_1^3(1+x_2+x_2^2)^2$};
   \end{tikzpicture}~\ ~
  \end{center}
  \begin{center}
   \begin{tikzpicture}
    \draw[step=0.7cm,color=black] (0,0) grid (3.5,1.4);
    \draw[color=gray,very thin] (0,0) -- (3.5,1.4);
    \draw[fill=black] (0,0) circle (1.5pt);
    \draw (0.35,0.2) node {$012$};
    \draw[fill=black] (0.7,0) circle (1.5pt);
    \draw (1.05,0.2) node {$012$};
    \draw[fill=black] (1.4,0) circle (1.5pt);
    \draw (1.75,0.2) node {$012$};
    \draw[fill=black] (2.1,0) circle (1.5pt);
    \draw (2.25,0.35) node {$0$};
    \draw[fill=black] (2.1,0.7) circle (1.5pt);
    \draw (2.45,0.9) node {$012$};
    \draw[fill=black] (2.8,0.7) circle (1.5pt);
    \draw (3.15,0.9) node {$0$};
    \draw[fill=black] (3.5,0.7) circle (1.5pt);
    \draw (3.65,1.05) node {$1$};
    \draw[fill=black] (3.5,1.4) circle (1.5pt);
    \draw (1.75,-0.25) node {$x_1(1+x_2+x_2^2)^4$};
   \end{tikzpicture}
   \hfill
   \begin{tikzpicture}
    \draw[step=0.7cm,color=black] (0,0) grid (3.5,1.4);
    \draw[color=gray,very thin] (0,0) -- (3.5,1.4);
    \draw[fill=black] (0,0) circle (1.5pt);
    \draw (0.35,0.2) node {$012$};
    \draw[fill=black] (0.7,0) circle (1.5pt);
    \draw (1.05,0.2) node {$012$};
    \draw[fill=black] (1.4,0) circle (1.5pt);
    \draw (1.75,0.2) node {$0$};
    \draw[fill=black] (2.1,0) circle (1.5pt);
    \draw (2.25,0.35) node {$1$};
    \draw[fill=black] (2.1,0.7) circle (1.5pt);
    \draw (2.45,0.9) node {$012$};
    \draw[fill=black] (2.8,0.7) circle (1.5pt);
    \draw (3.15,0.9) node {$0$};
    \draw[fill=black] (3.5,0.7) circle (1.5pt);
    \draw (3.65,1.05) node {$1$};
    \draw[fill=black] (3.5,1.4) circle (1.5pt);
    \draw (1.75,-0.25) node {$x_1^2(1+x_2+x_2^2)^3$};
   \end{tikzpicture}
   \hfill
   \begin{tikzpicture}
    \draw[step=0.7cm,color=black] (0,0) grid (3.5,1.4);
    \draw[color=gray,very thin] (0,0) -- (3.5,1.4);
    \draw[fill=black] (0,0) circle (1.5pt);
    \draw (0.35,0.2) node {$012$};
    \draw[fill=black] (0.7,0) circle (1.5pt);
    \draw (1.05,0.2) node {$0$};
    \draw[fill=black] (1.4,0) circle (1.5pt);
    \draw (1.75,0.2) node {$0$};
    \draw[fill=black] (2.1,0) circle (1.5pt);
    \draw (2.25,0.35) node {$2$};
    \draw[fill=black] (2.1,0.7) circle (1.5pt);
    \draw (2.45,0.9) node {$012$};
    \draw[fill=black] (2.8,0.7) circle (1.5pt);
    \draw (3.15,0.9) node {$0$};
    \draw[fill=black] (3.5,0.7) circle (1.5pt);
    \draw (3.65,1.05) node {$1$};
    \draw[fill=black] (3.5,1.4) circle (1.5pt);
    \draw (1.75,-0.25) node {$x_1^3(1+x_2+x_2^2)^2$};
   \end{tikzpicture}
   \hfill
   \begin{tikzpicture}
    \draw[step=0.7cm,color=black] (0,0) grid (3.5,1.4);
    \draw[color=gray,very thin] (0,0) -- (3.5,1.4);
    \draw[fill=black] (0,0) circle (1.5pt);
    \draw (0.35,0.2) node {$0$};
    \draw[fill=black] (0.7,0) circle (1.5pt);
    \draw (1.05,0.2) node {$0$};
    \draw[fill=black] (1.4,0) circle (1.5pt);
    \draw (1.75,0.2) node {$0$};
    \draw[fill=black] (2.1,0) circle (1.5pt);
    \draw (2.25,0.35) node {$3$};
    \draw[fill=black] (2.1,0.7) circle (1.5pt);
    \draw (2.45,0.9) node {$012$};
    \draw[fill=black] (2.8,0.7) circle (1.5pt);
    \draw (3.15,0.9) node {$0$};
    \draw[fill=black] (3.5,0.7) circle (1.5pt);
    \draw (3.65,1.05) node {$1$};
    \draw[fill=black] (3.5,1.4) circle (1.5pt);
    \draw (1.75,-0.25) node {$x_1^4(1+x_2+x_2^2)$};
   \end{tikzpicture}~\ ~
  \end{center}
  \begin{center}
   \begin{tikzpicture}
    \draw[step=0.7cm,color=black] (0,0) grid (3.5,1.4);
    \draw[color=gray,very thin] (0,0) -- (3.5,1.4);
    \draw[fill=black] (0,0) circle (1.5pt);
    \draw (0.35,0.2) node {$012$};
    \draw[fill=black] (0.7,0) circle (1.5pt);
    \draw (1.05,0.2) node {$012$};
    \draw[fill=black] (1.4,0) circle (1.5pt);
    \draw (1.75,0.2) node {$012$};
    \draw[fill=black] (2.1,0) circle (1.5pt);
    \draw (2.25,0.35) node {$0$};
    \draw[fill=black] (2.1,0.7) circle (1.5pt);
    \draw (2.45,0.9) node {$0$};
    \draw[fill=black] (2.8,0.7) circle (1.5pt);
    \draw (3.15,0.9) node {$0$};
    \draw[fill=black] (3.5,0.7) circle (1.5pt);
    \draw (3.65,1.05) node {$2$};
    \draw[fill=black] (3.5,1.4) circle (1.5pt);
    \draw (1.75,-0.25) node {$x_1^2(1+x_2+x_2^2)^3$};
   \end{tikzpicture}
   \hfill
   \begin{tikzpicture}
    \draw[step=0.7cm,color=black] (0,0) grid (3.5,1.4);
    \draw[color=gray,very thin] (0,0) -- (3.5,1.4);
    \draw[fill=black] (0,0) circle (1.5pt);
    \draw (0.35,0.2) node {$012$};
    \draw[fill=black] (0.7,0) circle (1.5pt);
    \draw (1.05,0.2) node {$012$};
    \draw[fill=black] (1.4,0) circle (1.5pt);
    \draw (1.75,0.2) node {$0$};
    \draw[fill=black] (2.1,0) circle (1.5pt);
    \draw (2.25,0.35) node {$1$};
    \draw[fill=black] (2.1,0.7) circle (1.5pt);
    \draw (2.45,0.9) node {$0$};
    \draw[fill=black] (2.8,0.7) circle (1.5pt);
    \draw (3.15,0.9) node {$0$};
    \draw[fill=black] (3.5,0.7) circle (1.5pt);
    \draw (3.65,1.05) node {$2$};
    \draw[fill=black] (3.5,1.4) circle (1.5pt);
    \draw (1.75,-0.25) node {$x_1^3(1+x_2+x_2^2)^2$};
   \end{tikzpicture}
   \hfill
   \begin{tikzpicture}
    \draw[step=0.7cm,color=black] (0,0) grid (3.5,1.4);
    \draw[color=gray,very thin] (0,0) -- (3.5,1.4);
    \draw[fill=black] (0,0) circle (1.5pt);
    \draw (0.35,0.2) node {$012$};
    \draw[fill=black] (0.7,0) circle (1.5pt);
    \draw (1.05,0.2) node {$0$};
    \draw[fill=black] (1.4,0) circle (1.5pt);
    \draw (1.75,0.2) node {$0$};
    \draw[fill=black] (2.1,0) circle (1.5pt);
    \draw (2.25,0.35) node {$2$};
    \draw[fill=black] (2.1,0.7) circle (1.5pt);
    \draw (2.45,0.9) node {$0$};
    \draw[fill=black] (2.8,0.7) circle (1.5pt);
    \draw (3.15,0.9) node {$0$};
    \draw[fill=black] (3.5,0.7) circle (1.5pt);
    \draw (3.65,1.05) node {$2$};
    \draw[fill=black] (3.5,1.4) circle (1.5pt);
    \draw (1.75,-0.25) node {$x_1^4(1+x_2+x_2^2)$};
   \end{tikzpicture}
   \hfill
   \begin{tikzpicture}
    \draw[step=0.7cm,color=black] (0,0) grid (3.5,1.4);
    \draw[color=gray,very thin] (0,0) -- (3.5,1.4);
    \draw[fill=black] (0,0) circle (1.5pt);
    \draw (0.35,0.2) node {$0$};
    \draw[fill=black] (0.7,0) circle (1.5pt);
    \draw (1.05,0.2) node {$0$};
    \draw[fill=black] (1.4,0) circle (1.5pt);
    \draw (1.75,0.2) node {$0$};
    \draw[fill=black] (2.1,0) circle (1.5pt);
    \draw (2.25,0.35) node {$3$};
    \draw[fill=black] (2.1,0.7) circle (1.5pt);
    \draw (2.45,0.9) node {$0$};
    \draw[fill=black] (2.8,0.7) circle (1.5pt);
    \draw (3.15,0.9) node {$0$};
    \draw[fill=black] (3.5,0.7) circle (1.5pt);
    \draw (3.65,1.05) node {$2$};
    \draw[fill=black] (3.5,1.4) circle (1.5pt);
    \draw (1.75,-0.25) node {$x_1^5$};
   \end{tikzpicture}~\ ~
  \end{center}
  \begin{center}
   \begin{tikzpicture}
    \draw[step=0.7cm,color=black] (0,0) grid (3.5,1.4);
    \draw[color=gray,very thin] (0,0) -- (3.5,1.4);
    \draw[fill=black] (0,0) circle (1.5pt);
    \draw (0.35,0.2) node {$012$};
    \draw[fill=black] (0.7,0) circle (1.5pt);
    \draw (1.05,0.2) node {$012$};
    \draw[fill=black] (1.4,0) circle (1.5pt);
    \draw (1.75,0.2) node {$01$};
    \draw[fill=black] (2.1,0) circle (1.5pt);
    \draw (2.25,0.35) node {$0$};
    \draw[fill=black] (2.1,0.7) circle (1.5pt);
    \draw (2.45,0.9) node {$0$};
    \draw[fill=black] (2.8,0.7) circle (1.5pt);
    \draw (3.15,0.9) node {$0$};
    \draw[fill=black] (3.5,0.7) circle (1.5pt);
    \draw (3.65,1.05) node {$3$};
    \draw[fill=black] (3.5,1.4) circle (1.5pt);
    \draw (1.9,-0.25) node {$x_1^3(1+x_2+x_2^2)^2(1+x_2)$};
   \end{tikzpicture}
   \hfill
   \begin{tikzpicture}
    \draw[step=0.7cm,color=black] (0,0) grid (3.5,1.4);
    \draw[color=gray,very thin] (0,0) -- (3.5,1.4);
    \draw[fill=red] (0,0) circle (1.5pt);
    \draw (0.35,0.2) node {$012$};
    \draw[fill=red] (0.7,0) circle (1.5pt);
    \draw (1.05,0.2) node {$01$};
    \draw[fill=red] (1.4,0) circle (1.5pt);
    \draw (1.75,0.2) node {$0$};
    \draw[fill=red] (2.1,0) circle (1.5pt);
    \draw (2.25,0.35) node {$1$};
    \draw[fill=red] (2.1,0.7) circle (1.5pt);
    \draw (2.45,0.9) node {$0$};
    \draw[fill=red] (2.8,0.7) circle (1.5pt);
    \draw (3.15,0.9) node {$0$};
    \draw[fill=red] (3.5,0.7) circle (1.5pt);
    \draw (3.65,1.05) node {$3$};
    \draw[fill=red] (3.5,1.4) circle (1.5pt);
    \draw (1.8,-0.25) node {$x_1^4(1+x_2+x_2^2)(1+x_2)$};
   \end{tikzpicture}
   \hfill
   \begin{tikzpicture}
    \draw[step=0.7cm,color=black] (0,0) grid (3.5,1.4);
    \draw[color=gray,very thin] (0,0) -- (3.5,1.4);
    \draw[fill=black] (0,0) circle (1.5pt);
    \draw (0.35,0.2) node {$01$};
    \draw[fill=black] (0.7,0) circle (1.5pt);
    \draw (1.05,0.2) node {$0$};
    \draw[fill=black] (1.4,0) circle (1.5pt);
    \draw (1.75,0.2) node {$0$};
    \draw[fill=black] (2.1,0) circle (1.5pt);
    \draw (2.25,0.35) node {$2$};
    \draw[fill=black] (2.1,0.7) circle (1.5pt);
    \draw (2.45,0.9) node {$0$};
    \draw[fill=black] (2.8,0.7) circle (1.5pt);
    \draw (3.15,0.9) node {$0$};
    \draw[fill=black] (3.5,0.7) circle (1.5pt);
    \draw (3.65,1.05) node {$3$};
    \draw[fill=black] (3.5,1.4) circle (1.5pt);
    \draw (1.75,-0.25) node {$x_1^5(1+x_2)$};
   \end{tikzpicture}
   \hfill
   \begin{tikzpicture}
    \draw[step=0.7cm,color=black] (0,0) grid (3.5,1.4);
    \draw[color=gray,very thin] (0,0) -- (3.5,1.4);
    \draw[fill=black] (0,0) circle (1.5pt);
    \draw (0.35,0.2) node {$0$};
    \draw[fill=black] (0.7,0) circle (1.5pt);
    \draw (1.05,0.2) node {$0$};
    \draw[fill=black] (1.4,0) circle (1.5pt);
    \draw (1.75,0.2) node {$0$};
    \draw[fill=black] (2.1,0) circle (1.5pt);
    \draw (2.25,0.35) node {$3$};
    \draw[fill=black] (2.1,0.7) circle (1.5pt);
    \draw (2.45,0.9) node {$0$};
    \draw[fill=black] (2.8,0.7) circle (1.5pt);
    \draw (3.15,0.9) node {$0$};
    \draw[fill=black] (3.5,0.7) circle (1.5pt);
    \draw (3.65,1.05) node {$3$};
    \draw[fill=black] (3.5,1.4) circle (1.5pt);
    \draw (1.75,-0.25) node {$x_1^6$};
   \end{tikzpicture}.
  \end{center}
  We will illustrate the compatibility of $S_1=(2,1,0,0,0)$ and $S_2=(1,3)$, where we have written
  \[S_1=\big(S_1(h_1),S_1(h_2),S_1(h_3),S_1(h_4),S_1(h_5)\big)\quad\text{ and }\quad S_2=\big(S_2(v_1),S_2(v_2)\big).\]
  As in Remark~\ref{rem:trivial_compatibility} the condition \eqref{eq:compatibility1} is automatically satisfied for the horizontal edges $h_3$, $h_4$, $h_5$ and each vertical edge.  The condition \eqref{eq:compatibility1} cannot be satisfied for $h_1$ and any vertical edge.  The condition \eqref{eq:compatibility2} is satisfied for $h_1$ and $v_1$ by taking $e=h_3$, while the condition \eqref{eq:compatibility2} is satisfied for $h_1$ and $v_2$ by taking $e=h_2$.  The condition \eqref{eq:compatibility1} cannot be satisfied for $h_2$ and $v_1$, but the condition \eqref{eq:compatibility2} is satisfied by taking $e=h_3$.  The condition \eqref{eq:compatibility1} is satisfied for $h_2$ and $v_2$ by taking $e=v_1$, while the condition \eqref{eq:compatibility2} cannot be satisfied.
 \end{example} 

 For a horizontal grading $S_1:D_1\to\ZZ_{\ge0}$, write $\cC(S_1)$ for the set of all vertical gradings $S_2:D_2\to\ZZ_{\ge0}$ with $(S_1,S_2)\in\cC$.  Define $\cC(S_2)$ similarly.  We define the \emph{magnitude} of a horizontal grading $S_1$ or a vertical grading $S_2$ respectively by $|S_1|=\sum\limits_{h\in D_1}S_1(h)$ and $|S_2|=\sum\limits_{v\in D_2} S_2(v)$.  The following result greatly simplifies the check for membership in $\cC(S_1)$ and $\cC(S_2)$.
 \begin{lemma}\label{le:shadows}\mbox{}
  \begin{enumeratea}
   \item For every $S_1:D_1\to\ZZ_{\ge0}$ there exist subsets $\rsh(S_1)\subset\sh(S_1)\subset D_2$ such that:
   \begin{enumeratei}
    \item for $S_2:D_2\to\ZZ_{\ge0}$, $S_2\in\cC(S_1)$ if and only if $S_2(v)=0$ for all $v\in\sh(S_1)\setminus\rsh(S_1)$ and one of the conditions \eqref{eq:compatibility1} or \eqref{eq:compatibility2} is satisfied for each $h\in\supp(S_1)$ and $v\in\rsh(S_1)$;
    \item $|\sh(S_1)|=\min(a_2,|S_1|)$.
   \end{enumeratei}
   \item For every $S_2:D_2\to\ZZ_{\ge0}$ there exist subsets $\rsh(S_2)\subset\sh(S_2)\subset D_1$ such that
   \begin{enumeratei}
    \item for $S_1:D_1\to\ZZ_{\ge0}$, $S_1\in\cC(S_2)$ if and only if $S_1(h)=0$ for all $h\in\sh(S_2)\setminus\rsh(S_2)$ and one of the conditions \eqref{eq:compatibility1} or \eqref{eq:compatibility2} is satisfied for each $h\in\rsh(S_2)$ and $v\in\supp(S_2)$;
    \item $|\sh(S_2)|=\min(a_1,|S_2|)$.
   \end{enumeratei}
  \end{enumeratea}
 \end{lemma}

 We will refer to the sets $\sh(S_i)$ and $\rsh(S_i)$ respectively as the \emph{shadow} and \emph{remote shadow} of $S_i$ ($i=1,2$).  Before proving Lemma~\ref{le:shadows} we require some preparation.  The shadow of each $S_i$ will be defined in terms of certain ``local shadows''.
 \begin{definition}\label{def:local}\mbox{}
  \begin{enumeratea}
   \item For each $h\in D_1$, define $\sh(h;S_1)$, the \emph{local shadow of $S_1$ at $h$}, as the set of vertical edges $(he)_2$ in the shortest subpath $he$ of $D$ such that $|(he)_2|=\sum\limits_{h'\in(he)_1}S_1(h')$.  Write $D(h;S_1)=he$ for this minimal path.  If there is no such subpath then we set $\sh(h;S_1)=D_2$ and $D(h;S_1)=D$.  We define the \emph{shadow} of $S_1$ as the union of local shadows: $\sh(S_1)=\bigcup\limits_{h\in D_1}\sh(h;S_1)$.
   \item For each $v\in D_2$, define $\sh(v;S_2)$, the \emph{local shadow of $S_2$ at $v$}, as the set of horizontal edges $(ev)_1$ in the shortest subpath $ev$ of $D$ such that $|(ev)_1|=\sum\limits_{v'\in(ev)_2}S_2(v')$.  Write $D(v;S_2)$ for this minimal path $ev$.  If there is no such subpath then we set $\sh(v;S_2)=D_1$ and $D(v;S_2)=D$.  We define the \emph{shadow} of $S_2$ as the union of local shadows: $\sh(S_2)=\bigcup\limits_{v\in D_2}\sh(v;S_2)$.
  \end{enumeratea}
 \end{definition}
 \begin{remark}
  When $S_1(h)=0$, we have $\sh(h;S_1)=\emptyset$ and $D(h;S_1)=hh$ is the subpath of $D$ consisting of only the edge $h$.  A similar claim holds if $S_2(v)=0$.
 \end{remark}

 It will be useful to consider certain \emph{shadow statistics} for subpaths of $D$ which are closely related to the compatibility conditions for a pair $(S_1,S_2)\in\cC$.  
 \begin{definition}\label{def:shadow statistics}
  For a subpath $ee'\subset D$ define the \emph{horizontal shadow statistic}
  \[f_{S_1}(ee'):=-|(ee')_2|+\sum\limits_{h\in(ee')_1}S_1(h)\]
  for each horizontal grading $S_1$ and the \emph{vertical shadow statistic}
  \[f_{S_2}(ee'):=-|(ee')_1|+\sum\limits_{v\in(ee')_2}S_2(v)\]
  for each vertical grading $S_2$.  
 \end{definition}
 It immediately follows from the definitions that the functions $f_{S_1}$ and $f_{S_2}$ satisfy the following additivity properties with respect to concatenation of paths:
 \begin{align*}
  f_{S_1}(ee''')&=f_{S_1}(ee')+f_{S_1}(e''e'''),\\
  f_{S_2}(ee''')&=f_{S_2}(ee')+f_{S_2}(e''e'''),
 \end{align*}
 whenever $ee'''=ee'\amalg e''e'''$.  The next lemma relates the minimal shadow paths $D(h;S_1)$ and $D(v;S_2)$ to the shadow statistics.
 \begin{lemma}\label{le:sub_compatibility}\mbox{}
  \begin{enumeratea}
   \item Suppose $h\in\supp(S_1)$ and write $D(h;S_1)=he$.
   \begin{enumeratei}
    \item $f_{S_1}(D(h;S_1))=0$.
    \item For any proper subpath $he'\subset D(h;S_1)$ we have $f_{S_1}(he')>0$.
    \item For any proper subpath $e'e\subset D(h;S_1)$ we have $f_{S_1}(e'e)<0$.
   \end{enumeratei}
   In particular, one may conclude that the last edge $e$ of $D(h;S_1)$ is vertical.
   \item Suppose $v\in\supp(S_2)$ and write $D(v;S_2)=ev$.
   \begin{enumeratei}
    \item $f_{S_2}(D(v;S_2))=0$.
    \item For any proper subpath $e'v\subset D(v;S_2)$ we have $f_{S_2}(e'v)>0$.
    \item For any proper subpath $ee'\subset D(v;S_2)$ we have $f_{S_2}(ee')<0$.
   \end{enumeratei}
   In particular, one may conclude that the first edge $e$ of $D(v;S_2)$ is horizontal.
  \end{enumeratea}
 \end{lemma}
 \begin{proof}
  We will prove (a), the proof of (b) is similar.  Recalling the definition of $D(h;S_1)$ establishes (i).  Consider the value of $f_{S_1}(he')$ as $e'$ traverses the path $D(h;S_1)$ from left to right.  When $e'=h$, we have $f_{S_1}(hh)=S_1(h)>0$ since $h\in\supp(S_1)$.  As $e'$ moves to the right, $f_{S_1}(he')$ will either increase, stay constant, or decrease by 1.  It follows that $f_{S_1}(he')$ will remain positive until it first reaches zero, at which point we must have $e'=e$ by minimality.  This proves (ii), (iii) is an immediate consequence of (i) and (ii) using the additivity property of $f_{S_1}$.
 \end{proof}
 
 The next easy result is an immediate consequence of Lemma~\ref{le:sub_compatibility} but will be useful in the proofs of Proposition~\ref{prop:horizontal grading bijection} and Proposition~\ref{prop:supports} below.
 \begin{corollary}\label{cor:remote_shadow_compatibility}
  Suppose $h\in\supp(S_1)$ and let $v\in D(h;S_1)$ be any vertical edge.  Then \eqref{eq:compatibility1} is not satisfied for $h$ and $v$, however for any $S_2\in\cC(S_1)$ the condition \eqref{eq:compatibility2} is satisfied for $h$ and $v$.  In particular, $D(v;S_2)$ is a proper subpath of $hv$ for each $S_2\in\cC(S_1)$.
 \end{corollary}
 \begin{proof}
  By Lemma~\ref{le:sub_compatibility} we have $f_{S_1}(he)>0$ for every edge $e\in h\overline{v}$ and thus \eqref{eq:compatibility1} cannot be satisfied.  Since $S_2\in\cC(S_1)$ the condition \eqref{eq:compatibility2} must be satisfied.
 \end{proof}

 We can now understand the relationships between the various local shadows associated to a horizontal grading $S_1$ or a vertical grading $S_2$.
 \begin{lemma}\label{le:shadow containment}\mbox{}
  \begin{enumeratea}
   \item Let $S_1$ be a horizontal grading.  For any horizontal edges $h,h'\in D_1$, the local shadows $\sh(h;S_1)$ and $\sh(h';S_1)$ are either disjoint or one is contained in the other.
   \item Let $S_2$ be a vertical grading.  For any vertical edges $v,v'\in D_2$, the local shadows $\sh(v;S_2)$ and $\sh(v';S_2)$ are either disjoint or one is contained in the other.
  \end{enumeratea}
 \end{lemma}
 \begin{proof}
  We will prove (1), the proof of (2) is similar.  The proof will be in terms of the local shadow paths $D(h;S_1)$ and $D(h';S_1)$.  

  For $h=h'$ the claim is trivial.  If either $D(h;S_1)=D$ or $D(h';S_1)=D$ the claim is again trivially true, thus we assume that both $D(h;S_1)$ and $D(h';S_1)$ are proper subpaths of $D$.  

  Suppose $D(h;S_1)$ and $D(h';S_1)$ intersect but not in a proper containment.  Then either $h\in D(h';S_1)$ or $h'\in D(h;S_1)$, without loss of generality assume $h'\in D(h;S_1)=he$.  Then $e\in D(h';S_1)$ since there is not a proper containment of subpaths.  Denote by $e'$ the edge of $D$ immediately preceding $h'$.  It follows from Lemma~\ref{le:sub_compatibility} that $f_{S_1}(he')>0$ and $f_{S_1}(h'e)>0$.  But notice that the path $he$ is just the concatenation of $he'$ and $h'e$, in particular adding the two inequalities above gives $f_{S_1}(he)=f_{S_1}(he')+f_{S_1}(h'e)>0$, contradicting the definition of $D(h;S_1)$.
 \end{proof}

 \begin{definition}\mbox{}
  \begin{enumeratea}
   \item The \emph{remote shadow} of $S_1$ is the set $\rsh(S_1)$ obtained from $\sh(S_1)$ by removing for each $d\in[1,a_1]$ the (up to) $S_1(h_d)$ vertical edges of depth $d$ immediately following $h_d$.  By considering the definition of compatibility one may alternatively described $\rsh(S_1)\subset\sh(S_1)$ as the subset consisting of all vertical edges $v$ for which there exists a vertical grading $S_2\in\cC(S_1)$ with $S_2(v)\ne0$.
   \item The \emph{remote shadow} of $S_2$ is the set $\rsh(S_2)$ obtained from $\sh(S_2)$ by removing for each $\ell\in[1,a_2]$ the (up to) $S_2(v_\ell)$ horizontal edges of height $\ell-1$ immediately preceding $v_\ell$.  By considering the definition of compatibility one may alternatively describe $\rsh(S_2)\subset\sh(S_2)$ as the subset consisting of all horizontal edges $h$ for which there exists a horizontal grading $S_1\in\cC(S_2)$ with $S_1(h)\ne0$.
  \end{enumeratea}
 \end{definition}
   
 We are now ready to prove Lemma~\ref{le:shadows}.  We will explain part (a), part (b) follows by a similar argument.
 \begin{proof}[Proof of Lemma~\ref{le:shadows}.a]
  Fix a horizontal grading $S_1$.  We have constructed the sets $\rsh(S_1)$ and $\sh(S_1)$.  Thus it remains to check that they satisfy the desired conditions (i) and (ii).  Suppose $S_2\in\cC(S_1)$.  From the definition of $\rsh(S_1)$ we must have $S_2(v)=0$ for all $v\in\sh(S_1)\setminus\rsh(S_1)$.  Since $S_1$ and $S_2$ are compatible we see that either \eqref{eq:compatibility1} or \eqref{eq:compatibility2} is satisfied for every pair $h\in D_1$ and $v\in D_2$, in particular for $h\in\supp(S_1)$ and $v\in\rsh(S_1)$.  This proves the forward implication of (i).
  
  Now suppose $S_2:D_2\to\ZZ_{\ge0}$ satisfies $S_2(v)=0$ for all $v\in\sh(S_1)\setminus\rsh(S_1)$ and one of the conditions \eqref{eq:compatibility1} or \eqref{eq:compatibility2} is satisfied for each $h\in\supp(S_1)$ and $v\in\rsh(S_1)$.  As in Remark~\ref{rem:trivial_compatibility} we know that, when $S_1(h)=0$, \eqref{eq:compatibility1} is satisfied for $h$ and each $v\in D_2$.  Thus to check compatibility we may restrict attention to $h\in\supp(S_1)$. 

  Consider the path $D(h;S_1)=he$.  Using the definition of $\sh(S_1)$ as a union of local shadows we see that any $v\in D_2\setminus\sh(S_1)$ lies outside the path $he$, in other words the edge $e$ is contained in the path $hv$.  Thus \eqref{eq:compatibility1} is satisfied for $h$ and $v$ using exactly this edge $e$.  Thus we have further reduced the check for compatibility to $h\in\supp(S_1)$ and $v\in\sh(S_1)$.  

  Furthermore since $S_2(v)=0$ for $v\in\sh(S_1)\setminus\rsh(S_1)$, it again follows from Remark~\ref{rem:trivial_compatibility} that \eqref{eq:compatibility2} is trivially satisfied for all $h\in\supp(S_1)$ and $v\in\sh(S_1)\setminus\rsh(S_1)$.  Thus we only need one of the conditions \eqref{eq:compatibility1} or \eqref{eq:compatibility2} for $h\in\supp(S_1)$ and $v\in\rsh(S_1)$, but this is guaranteed by assumption.  Thus $S_2\in\cC(S_1)$, establishing the reverse implication of (i).

  We now move on to proving (ii).  If there exists an edge $h\in D_1$ with $\sh(h;S_1)=D_2$ then $f_{S_1}(he)\ge0$ for every edge $e$.  In particular, when $e$ is the edge immediately preceding $h$ in $D$ we have $|\sh(S_1)|=a_2=|(he)_2|\le\sum\limits_{h'\in(he)_1}S_1(h')=|S_1|$.  This establishes (ii) in this case.  

  Now suppose $\sh(h;S_1)\ne D_2$ for all $h\in D_1$.  It follows from Lemma~\ref{le:shadow containment} that for any fixed horizontal grading $S_1$ we can find horizontal edges $\eta_1,\ldots,\eta_p\in\supp(S_1)$ (labeled in the natural order along the Dyck path $D$) such that
 \begin{itemize}
  \item $\sh(\eta_i;S_1)\cap\sh(\eta_j;S_1)=\emptyset$ for $i\ne j$;
  \item $\sh(S_1)=\sh(\eta_1;S_1)\amalg\cdots\amalg\sh(\eta_p;S_1)$ is a partition of $\sh(S_1)$.
 \end{itemize}
 Define vertical edges $\nu_1,\ldots,\nu_p$ by $D(\eta_j;S_1)=\eta_j\nu_j$.  Then each horizontal edge $h\in\supp(S_1)$ is contained in one of the maximal local shadows $\sh(\eta_j;S_1)$ and thus we see that 
  \[|\sh(S_1)|=|\sh(\eta_1;S_1)|+\cdots+|\sh(\eta_p;S_1)|=\sum\limits_{j=1}^p|(\eta_j\nu_j)_2|=\sum\limits_{j=1}^p\sum\limits_{h'\in(\eta_j\nu_j)_1}S_1(h')=|S_1|,\]
  Since $\sh(S_1)\subset D_2$, we have $|S_1|=|\sh(S_1)|\le a_2$.  This completes the proof of (ii).
 \end{proof}
 
 It will be useful to partition the remote shadows according to which local shadow contains a given edge. 
 \begin{definition}\label{def:remote_shadows}\mbox{}
  \begin{enumeratea}
   \item For $0<j\le a_1$ and $0<d\le a_1$, denote by $\rsh(S_1)_{j;d}$ the set of $v\in\rsh(S_1)$ of depth $d$ such that $v\in\sh(h_j;S_1)$ and $h_j$ is the first edge before $v$ with this property (i.e. the path $h_jv$ is shortest possible).  Define the \emph{local remote shadow} of the edge $h_j$ as $\rsh(h_j;S_1)=\coprod\limits_{d\in[1,a_1]}\rsh(S_1)_{j;d}$.
   \item For $0<k\le a_2$ and $0\le \ell<a_2$, denote by $\rsh(S_2)_{k;\ell}$ the set of $h\in\rsh(S_2)$ of height $\ell$ such that $h\in\sh(v_k;S_2)$ and $v_k$ is the first edge after $h$ with this property (i.e. the path $hv_k$ is shortest possible).  Define the \emph{local remote shadow} of the edge $v_k$ as $\rsh(v_k;S_2)=\coprod\limits_{\ell\in[0,a_2-1]}\rsh(S_2)_{k;\ell}$.
  \end{enumeratea}
 \end{definition}
 \begin{remark}\label{rem:remote_shadow_decomposition}
  For a horizontal grading $S_1$ there are only finitely many $j$ with $\rsh(h_j;S_1)\ne\emptyset$.  Thus there exists a sequence $1\le j_1<\cdots<j_n\le a_1$ such that $\rsh(h_{j_t};S_1)\ne\emptyset$ for each $t\in[1,n]$ and there is a partition $\rsh(S_1)=\rsh(h_{j_1};S_1)\amalg\cdots\amalg\rsh(h_{j_n};S_1)$.
 \end{remark}

 The next lemma gives precise conditions describing when the remote shadows $\rsh(S_1)_{j;d}$ and $\rsh(S_2)_{j;\ell}$ are non-empty.  First note that $\rsh(S_1)_{d;d}=\emptyset$ for every $d$ and $\rsh(S_2)_{\ell+1;\ell}=\emptyset$ for every $\ell$.
 \begin{lemma}\label{le:remote_shadow_conditions}\mbox{}
  \begin{enumeratea}
   \item Let $0<d,j\le a_1$ with $j\ne d$.  Then $\rsh(S_1)_{j;d}\ne\emptyset$ if and only if $f_{S_1}(h\overline{h}_{d+1})<0<f_{S_1}(h_j\overline{h})$ for every horizontal edge $h\in (\overline{h}_j\overline{h}_{d+1})_1$.
   \item Let $0\le\ell< a_2$ and $0<j\le a_2$ with $j\ne\ell+1$.  Then $\rsh(S_2)_{j;\ell}\ne\emptyset$ if and only if $f_{S_2}(\overline{v}_\ell v)<0<f_{S_2}(\overline{v}v_j)$ for every vertical edge $v\in (\overline{v}_\ell\overline{v}_j)_2$.
  \end{enumeratea}
 \end{lemma}
 \begin{proof}
  We will establish (a), the proof of (b) is similar.

  If $\rsh(S_1)_{j;d}\ne\emptyset$ then $h_jh_d$ is a proper subpath of $D(h_j;S_1)$.  So Lemma~\ref{le:sub_compatibility} implies $0<f_{S_1}(h_je)$ for every edge $e\in h_jh_d$, taking $e$ to be horizontal gives half of the forward implication.

  Let $v\in\rsh(S_1)_{j;d}$ and suppose for sake of contradiction that there exists a horizontal edge $h\in(\overline{h}_j\overline{h}_{d+1})_1$ such that $f_{S_1}(h\overline{h}_{d+1})\ge0$.  Let $\ell\le d$ be the largest so that $f_{S_1}(h_\ell\overline{h}_{d+1})\ge0$.  If $h_\ell\notin\supp(S_1)$ then $f_{S_1}(h_{\ell+1}\overline{h}_{d+1})\ge f_{S_1}(h_\ell\overline{h}_{d+1})\ge0$ contradicting the maximality of $h_\ell$.  Thus we must have $S_1(h_\ell)\ne0$, in particular $D(h_\ell;S_1)$ is a non-trivial subpath of $D$.  If $h_\ell\overline{h}_{d+1}$ is a subpath of $D(h_\ell;S_1)$, then $v\in D(h_\ell;S_1)$ contradicting the minimality of the path $h_jv$.  Thus $D(h_\ell;S_1)$ must be a proper subpath of $h_\ell\overline{h}_{d+1}$.  
  
  Write $e'$ for the first edge after $D(h_\ell;S_1)$, since $f_{S_1}\big(D(h_\ell;S_1)\big)=0$ additivity gives $f_{S_1}(e'\overline{h}_{d+1})\ge0$.  If $D(h_\ell;S_1)$ does not contain $h_d$, we may move to the right from $e'$ until reaching a horizontal edge $h_{\ell'}$ which will satisfy $f_{S_1}(h_{\ell'}\overline{h}_{d+1})\ge f_{S_1}(e'\overline{h}_{d+1})\ge0$, this again contradicts the maximality of $h_\ell$.  Thus $D(h_\ell;S_1)$ must contain $h_d$.  But then every edge in the path $e'\overline{h}_{d+1}$ is vertical which is absurd since $f_{S_1}(e'\overline{h}_{d+1})\ge0$.  This contradiction shows that we must have $f_{S_1}(h\overline{h}_{d+1})<0$ for every horizontal edge $h\in(\overline{h}_j\overline{h}_{d+1})_1$.\\

  Now suppose $\rsh(S_1)_{j;d}=\emptyset$.  This can happen in one of two ways:
  \begin{itemize}
   \item the local shadow $\sh(h_j;S_1)$ is ``too short'';
   \item for each $v\in\sh(h_j;S_1)$ of depth $d$ there exists $j'=j'(v)$ so that the edge $h_{j'}$ is closer to $v$ than $h_j$ with $v\in\sh(h_{j'};S_1)$.
  \end{itemize} 

  The local shadow $\sh(h_j;S_1)$ is ``too short'' if $\sh(h_j;S_1)$ does not contain any vertical edges of depth $d$.  Then we have $f_{S_1}(h_je)=0$ where $D(h_j;S_1)=h_je$ with $e$ a vertical edge to the left of $h_d$.  Taking $h_{j'}\in h_jh_d$ to be the first horizontal edge to the right of $e$ will give $f_{S_1}(h_j\overline{h}_{j'})\le0$, establishing the reverse implication in this case.  

  Thus we may assume for each $v\in\sh(h_j;S_1)$ of depth $d$ the existence of a number $j'>j$ such that $v\in\sh(h_{j'};S_1)$.  Take $v$ to be the last edge of depth $d$ in $\sh(h_j;S_1)$.  By Lemma~\ref{le:sub_compatibility}.a.ii we have $f_{S_1}(h_j\overline{h}_{j'})>0$ and $f_{S_1}(h_{j'}v)\ge0$, so by additivity we have $f_{S_1}(h_jv)>0$.  This inequality implies that a vertical edge of depth $d$ which comes after $v$ is also contained in $\sh(h_j;S_1)$, but such cannot exist by the choice of $v$.  This implies $h_{j'}\overline{h}_{d+1}=h_{j'}v$ and so $f_{S_1}(h_{j'}\overline{h}_{d+1})=f_{S_1}(h_{j'}v)\ge0$, completing the proof of the reverse implication.
 \end{proof}
 \begin{remark}
  The inequalities above imply that when $\rsh(S_1)_{j;d}\ne\emptyset$ the set $\rsh(S_1)_{j';d'}$ will be empty whenever $j<j'<d<d'$.
 \end{remark}

 Now we are able to give an explicit formula for the sizes of the remote shadows.
 \begin{lemma}\label{le:remote_shadow_sizes}\mbox{}
  \begin{enumeratea}
   \item Let $0<d,j\le a_1$ with $d\ne j$ and suppose $f_{S_1}(h\overline{h}_{d+1})<0<f_{S_1}(h_j\overline{h})$ for every horizontal edge $h\in (\overline{h}_j\overline{h}_{d+1})_1$.  Then $|\rsh(S_1)_{j;d}|=\min\limits_{h\in (\overline{h}_j\overline{h}_{d+1})_1}\min\big(-f_{S_1}(h\overline{h}_{d+1}),f_{S_1}(h_j\overline{h})\big)$.
   \item Let $0\le\ell< a_2$ and $0<j\le a_2$ with $\ell\ne j$ and suppose $f_{S_2}(\overline{v}_\ell v)<0<f_{S_2}(\overline{v}v_j)$ for every vertical edge $v\in (\overline{v}_\ell\overline{v}_j)_2$.  Then $|\rsh(S_2)_{j;\ell}|=\min\limits_{v\in (\overline{v}_\ell\overline{v}_j)_2}\min\big(-f_{S_2}(\overline{v}_\ell v),f_{S_2}(\overline{v}v_j)\big)$.
  \end{enumeratea}
 \end{lemma}
 \begin{proof}
  As always, we will prove (a) and leave the adaptation of the proof to (b) for the reader.

  By Lemma~\ref{le:remote_shadow_conditions} the remote shadow $\rsh(S_1)_{j;d}$ is non-empty.  Consider any element $v\in\rsh(S_1)_{j;d}$.  Since $v$ has depth $d$ we have \[(h_j\overline{h}_d)_2\le(h_jv)_2\le(h_j\overline{h}_{d+1})_2.\]
  For each $h\in (\overline{h}_j\overline{h}_{d+1})_1$ we have $v\notin\sh(h;S_1)$ but $v\in\sh(h_j;S_1)$ so that $f_{S_1}(hv)<0$ and $f_{S_1}(h_jv)\ge0$.  Expanding the definition of $f_{S_1}$ gives 
  \[(h_j\overline{h})_2+\sum\limits_{h'\in(h\overline{h}_{d+1})_1}S_1(h')<(h_j\overline{h})_2+(hv)_2=(h_jv)_2\le \sum\limits_{h'\in(h_j\overline{h}_{d+1})_1}S_1(h')\]
  for every $h\in (\overline{h}_j\overline{h}_{d+1})_1$.  Combining the two displayed sequences of inequalities above we get
  \[\max\limits_{h\in (\overline{h}_j\overline{h}_{d+1})_1}\left((h_j\overline{h})_2+\sum\limits_{h'\in(h\overline{h}_{d+1})_1}S_1(h')\right)<(h_jv)_2\le\min\left((h_j\overline{h}_{d+1})_2\ ,\sum\limits_{h'\in(h_j\overline{h}_{d+1})_1}S_1(h')\right)\]
  for each $v\in\rsh(S_1)_{j;d}$.  Thus the number of possible $v\in\rsh(S_1)_{j;d}$ is equal to
  \begin{align*}
   |\rsh(S_1)_{j;d}|
   &=\min\left((h_j\overline{h}_{d+1})_2\ ,\sum\limits_{h'\in(h_j\overline{h}_{d+1})_1}S_1(h')\right)-\max\limits_{h\in (\overline{h}_j\overline{h}_{d+1})_1}\left((h_j\overline{h})_2+\sum\limits_{h'\in(h\overline{h}_{d+1})_1}S_1(h')\right)\\
   &=\min\left((h_j\overline{h}_{d+1})_2\ ,\sum\limits_{h'\in(h_j\overline{h}_{d+1})_1}S_1(h')\right)+\min\limits_{h\in (\overline{h}_j\overline{h}_{d+1})_1}\left(-(h_j\overline{h})_2-\sum\limits_{h'\in(h\overline{h}_{d+1})_1}S_1(h')\right)\\
   &=\min\limits_{h\in (\overline{h}_j\overline{h}_{d+1})_1}\min\left((h\overline{h}_{d+1})_2-\sum\limits_{h'\in(h\overline{h}_{d+1})_1}S_1(h')\ ,\ -(h_j\overline{h})_2+\sum\limits_{h'\in(h_j\overline{h})_1}S_1(h')\right)\\
   &=\min\limits_{h\in (\overline{h}_j\overline{h}_{d+1})_1}\min\left(-f_{S_1}(h\overline{h}_{d+1}),f_{S_1}(h_j\overline{h})\right).
  \end{align*}
 \end{proof}

 \subsection{Bounded Gradings}\label{sec:bounded}  Fix a vertical grading $S_2$ and suppose $r\ge\max\{S_2(v):v\in D_2\}\cup\{\left\lceil\frac{a_1}{a_2}\right\rceil\}$.  Write $D'=D_{ra_2-a_1,a_2}$ with horizontal edges $D_1'=\{h'_1,\ldots,h'_{ra_2-a_1}\}$ and vertical edges $D_2'=\{v'_1,\ldots,v'_{a_2}\}$.  We identify the vertical edges of $D$ and the vertical edges of $D'$ via a map $\varphi_r:D_2'\to D_2$ given by $\varphi_r(v'_j)=v_{a_2+1-j}$.  Define the vertical grading $\varphi^*_rS_2:D_2'\to\ZZ_{\ge0}$ by the rule $\varphi^*_rS_2(v'_j)=r-S_2(v_{a_2+1-j})$, it is well-defined (non-negative) since $S_2(v_j)$ is always bounded above by $r$.

 \begin{lemma}\label{le:f_and_varphi}
  For any $v'_i,v'_j\in D_2'$ we have $f_{\varphi^*_rS_2}(\overline{v'}_iv'_j)=-f_{S_2}(\overline{v}_{a_2-j}v_{a_2-i})$.
 \end{lemma}
 \begin{proof}
  This follows from a direct calculation using the definitions and Lemma~\ref{le:height_depth}:
  \begin{align*}
   f_{\varphi^*_rS_2}(\overline{v'}_iv'_j)
   &=\sum\limits_{v'\in(\overline{v'}_iv'_j)_2}\varphi^*_rS_2(v')-|(\overline{v'}_iv'_j)_1|\\
   &=\sum\limits_{v'\in(\overline{v'}_iv'_j)_2}\varphi^*_rS_2(v')-\big(\left\lceil j(ra_2-a_1)/a_2\right\rceil-\left\lceil i(ra_2-a_1)/a_2\right\rceil\big)\\
   &=\sum\limits_{v\in(v_{a_2+1-j}\overline{v}_{a_2+1-i})_2}\big(r-S_2(v)\big)-\left\lceil jr-ja_1/a_2\right\rceil+\left\lceil ir-ia_1/a_2\right\rceil\\
   &=-\sum\limits_{v\in(v_{a_2+1-j}\overline{v}_{a_2+1-i})_2}S_2(v)-\left\lceil -ja_1/a_2\right\rceil+\left\lceil -ia_1/a_2\right\rceil\\
   &=-\sum\limits_{v\in(\overline{v}_{a_2-j}v_{a_2-i})_2}S_2(v)+\left\lceil (a_2-i)a_1/a_2\right\rceil-\left\lceil (a_2-j)a_1/a_2\right\rceil\\
   &=-\sum\limits_{v\in(\overline{v}_{a_2-j}v_{a_2-i})_2}S_2(v)+|(\overline{v}_{a_2-j}v_{a_2-i})_1|\\
   &=-f_{S_2}(\overline{v}_{a_2-j}v_{a_2-i}).
  \end{align*}
 \end{proof}

 \begin{corollary}
  Suppose $0\le\ell< a_2$ and $0<j\le a_2$ with $\ell\ne j$.  Then we have 
  \[|\rsh(S_2)_{j;\ell}|=|\rsh(\varphi^*_rS_2)_{a_2-\ell;a_2-j}|.\]
 \end{corollary}
 \begin{proof}
  According to Lemma~\ref{le:remote_shadow_sizes} and Lemma~\ref{le:f_and_varphi} we have
  \begin{align*}
   |\rsh(S_2)_{j;\ell}|
   &=\min\limits_{v\in (\overline{v}_\ell\overline{v}_j)_2}\min\left(-f_{S_2}(\overline{v}_\ell v),f_{S_2}(\overline{v}v_j)\right)\\
   &=\min\limits_{v'\in (\overline{v'}_{a_2-j}\overline{v'}_{a_2-\ell})_2}\min\left(f_{\varphi^*_rS_2}(\overline{v'}v'_{a_2-\ell}),-f_{\varphi^*_rS_2}(\overline{v'}_{a_2-j}v')\right)\\
   &=|\rsh(\varphi^*_rS_2)_{a_2-\ell;a_2-j}|.
  \end{align*}
 \end{proof}

 Thus for each $\ell$ and $j$ we may define a bijection $\theta_{j;\ell}:\rsh(S_2)_{j;\ell}\to\rsh(\varphi^*_rS_2)_{a_2-\ell;a_2-j}$ which preserves the natural left-to-right order.  Following Lemma~\ref{le:shadows} we will define $\cC_{rs}(S_2)\subset\cC(S_2)$ to be the subset of horizontal gradings $S_1$ compatible with $S_2$ whose support is contained in the remote shadow of $S_2$, i.e. $\supp(S_1)\subset\rsh(S_2)$.  Now we can define a map $\Omega:\cC_{rs}(S_2)\to\cC_{rs}(\varphi_r^*S_2)$ as follows:
 \begin{equation}
 \label{eq:omega definition}
  \Omega(S_1)(h')=\begin{cases}
                          0 & \text{ if $h'\in D_1'\setminus\rsh(\varphi^*_rS_2)$;}\\
                          S_1(h) & \text{ if $h'=\theta_{j;\ell}(h)$ for $h\in\rsh(S_2)_{j;\ell}$.}
                         \end{cases}
 \end{equation}
 Note that $\Omega$ admits an obvious inverse map.  The following result is as much of an analogue of Lemma~\ref{le:f_and_varphi} as one can hope for, however it is the essential ingredient to show that $\Omega$ is well-defined, i.e. that the pair $(\Omega(S_1),\varphi^*_rS_2)$ is compatible.
 \begin{lemma}\label{le:f_and_Omega}
  Suppose $h'=\theta_{j;\ell}(h)$ for $h\in\rsh(S_2)_{j;\ell}\cap\supp(S_1)$. Then $f_{\Omega(S_1)}(h'v'_{a_2-\ell})=f_{S_1}(hv_j)$.
 \end{lemma}
 \begin{proof}
  Since $h\in\rsh(S_2)_{j;\ell}$ it follows from Lemma~\ref{le:shadow containment} that each horizontal edge in $hv_j$ is contained in some $\rsh(S_2)_{j';\ell'}$ where $\ell\le\ell'<j'\le j$.  Using this we may partition the edges in $\supp(S_1)\cap(hv_j)_1$ to get:
  \begin{align*}
   f_{S_1}(hv_j)
   &=\sum\limits_{h''\in(hv_j)_1}S_1(h'')-(hv_j)_2\\
   &=\sum\limits_{h''\in(hv_{\ell+1})_1}S_1(h'')+\sum\limits_{h''\in(\overline{v}_{\ell+1}v_j)_1}S_1(h'')-(v_{\ell+1}v_j)_2\\
   &=\sum\limits_{h''\in(hv_{\ell+1})_1\cap\rsh(S_2)_{j;\ell}}\!\!\!\!\!\! S_1(h'')+\sum\limits_{\substack{j'\st\\  \ell< j'<j}}\sum\limits_{h''\in\rsh(S_2)_{j';\ell}}\!\! S_1(h'')+\sum\limits_{\substack{\ell',j'\st\\ \ell<\ell'<j'\le j}}\sum\limits_{h''\in\rsh(S_2)_{j';\ell'}}\!\! S_1(h'')-(v_{\ell+1}v_j)_2.
  \end{align*}
  Applying the definition of $\Omega$ and the same logic as above this becomes
  \begin{align*}
   &\sum\limits_{h''\in(h'v'_{a_2+1-j})_1\cap\rsh(\varphi_r^*S_2)_{a_2-\ell;a_2-j}}\Omega(S_1)(h'')+\sum\limits_{\substack{j'\st\\ \ell<j'<j}}\sum\limits_{h''\in\rsh(\varphi_r^*S_2)_{a_2-\ell;a_2-j'}}\Omega(S_1)(h'')\\
   &\quad\quad+\sum\limits_{\substack{\ell',j'\st\\ \ell<\ell'<j'\le j}}\sum\limits_{h''\in\rsh(\varphi_r^*S_2)_{a_2-\ell';a_2-j'}}\Omega(S_1)(h'')-(v'_{a_2+1-j}v'_{a_2-\ell})_2\\
   &\quad\quad=\sum\limits_{h''\in(h'v'_{a_2+1-j})_1}\Omega(S_1)(h'')+\sum\limits_{h''\in(\overline{v'}_{a_2+1-j}v'_{a_2-\ell})_1}\Omega(S_1)(h'')-(v'_{a_2+1-j}v'_{a_2-\ell})_2\\
   &\quad\quad=\sum\limits_{h''\in(h'v'_{a_2-\ell})_1}\Omega(S_1)(h'')-(h'v'_{a_2-\ell})_2\\
   &\quad\quad=f_{\Omega(S_1)}(h'v'_{a_2-\ell}).
  \end{align*}
 \end{proof}
 \begin{proposition}
 \label{prop:horizontal grading bijection}
  Let $S_1:D_1\to\ZZ_{\ge0}$ be a horizontal grading.  Then $S_1\in\cC_{rs}(S_2)$ if and only if $\Omega(S_1)\in\cC_{rs}(\varphi_r^*S_2)$.
 \end{proposition}
 \begin{proof}
  We will prove the forward implication.  The reverse implication can be obtained by a similar argument with $\Omega^{-1}$.
  
  To check compatibility it suffices to consider $h'\in\supp(\Omega(S_1))\subset\rsh(\varphi_r^*S_2)$.  Suppose $h'=\theta_{j;\ell}(h)$ for $h\in\rsh(S_2)_{j;\ell}\cap\supp(S_1)$.  The containment $h'\in\rsh(\varphi_r^*S_2)_{a_2-\ell;a_2-j}$ implies for any vertical edge $v'_{a_2-\ell'}$ with $\ell<\ell'\le j$ that $D(v'_{a_2-\ell'};\varphi_r^*S_2)$ is a proper subpath of $h'v'_{a_2-\ell'}$, i.e. the condition \eqref{eq:compatibility2} is satisfied for $h'$ and $v'_{a_2-\ell'}$.  We claim that $D(h';\Omega(S_1))$ is a subpath of $h'v'_{a_2-\ell}$ which implies the condition \eqref{eq:compatibility1} is satisfied for $h'$ and each other vertical edge $v'_{a_2-\ell'}$ with $\ell'\le\ell$ or $\ell'>j$.

  Indeed, since $h\in D(v_j;S_2)$ each horizontal edge $h''\in(hv_j)_1$ is also contained in $D(v_j;S_2)$.  Then Corollary~\ref{cor:remote_shadow_compatibility} implies we may partition $hv_j$ into shadow paths $D(h'';S_1)$, horizontal edges outside the support of $S_1$, and the remaining vertical edges.  Thus we have $f_{\Omega(S_1)}(h'v'_{a_2-\ell})=f_{S_1}(hv_j)\le0$.  In particular, arguing as in the proof of Lemma~\ref{le:sub_compatibility} we see that there exists $v'\in(h'v'_{a_2-\ell})_2$ so that $f_{\Omega(S_1)}(h'v')=0$, i.e. $D(h';\Omega(S_1))$ is a subpath of $h'v'_{a_2-\ell}$ as desired.
 \end{proof}
 
 \subsection{Magnitudes of Bounded Gradings}
 Fix nonnegative integers $d_1,d_2\ge0$.  In this section we make explicit the possibilities for the pair of magnitudes $(|S_1|,|S_2|)$ associated to a compatible pair $(S_1,S_2)\in\cC_{a_1,a_2}$.  We assume throughout that $S_1(h)\le d_1$ for each $h\in D_1$ and $S_2(v)\le d_2$ for each $v\in D_2$.
 
 \begin{lemma}
 \label{le:grading bound}
 For any compatible pair $(S_1,S_2)\in\cC_{a_1,a_2}$ we must have $|S_1|<a_2$ or $|S_2|<a_1$.
 \end{lemma}  
 \begin{proof}
  Indeed, we will show that $|S_1|\ge a_2$ and $|S_2|\ge a_1$ is impossible.  Thus we assume that $|S_1|\ge a_2$ and will deduce $|S_2|<a_1$, the other case can be proven by a similar argument.
  
  Notice that under this assumption Lemma~\ref{le:shadows} says $|\sh(S_1)|=a_2$, in other words $\sh(S_1)=D_1$.  It follows that each vertical edge $v\in D_2$ is contained in some maximal local shadow.   Let $\sh(h;S_1)$ be such a maximal local shadow.  By Corollary~\ref{cor:remote_shadow_compatibility}, for each vertical edge $v\in\sh(h;S_1)$ the maximal shadow path $D(v;S_2)$ is a proper subpath of $hv$.  It follows that $h\notin\sh(S_2)$ and thus the inequality $|S_2|<a_1$ follows from Lemma~\ref{le:shadows}.
 \end{proof}
 
 \begin{proposition}
 \label{prop:supports}
  Let $(S_1,S_2)\in\cC_{a_1,a_2}$ be a bounded compatible pair.
  \begin{enumeratea}
   \item If $0\le d_2a_2\le a_1$, then $(|S_1|,|S_2|)$ is contained in the (possibly degenerate) trapezoid with corner vertices $(0,0)$, $(d_1a_1,0)$, $(0,d_2a_2)$, and $(d_1a_1-d_1d_2a_2,d_2a_2)$.
   \item If $0\le d_1a_1\le a_2$, then $(|S_1|,|S_2|)$ is contained in the (possibly degenerate) trapezoid with corner vertices $(0,0)$, $(d_1a_1,0)$, $(0,d_2a_2)$, and $(d_1a_1,d_2a_2-d_1d_2a_1)$.
   \item If $0<a_1<d_2a_2$ and $0<a_2<d_1a_1$, then $(|S_1|,|S_2|)$ is contained in the non-convex quadrilateral with corner vertices $(0,0)$, $(d_1a_1,0)$, $(0,d_2a_2)$, and $(a_2,a_1)$ with the convention that the boundary segments $\big[(0,0),(d_1a_1,0)\big]$ and $\big[(0,0),(0,d_2a_2)\big]$ are included, while the boundary segments $\big((d_1a_1,0),(a_2,a_1)\big]$ and $\big((0,d_2a_2),(a_2,a_1)\big]$ are excluded.
  \end{enumeratea}
 \end{proposition}
 
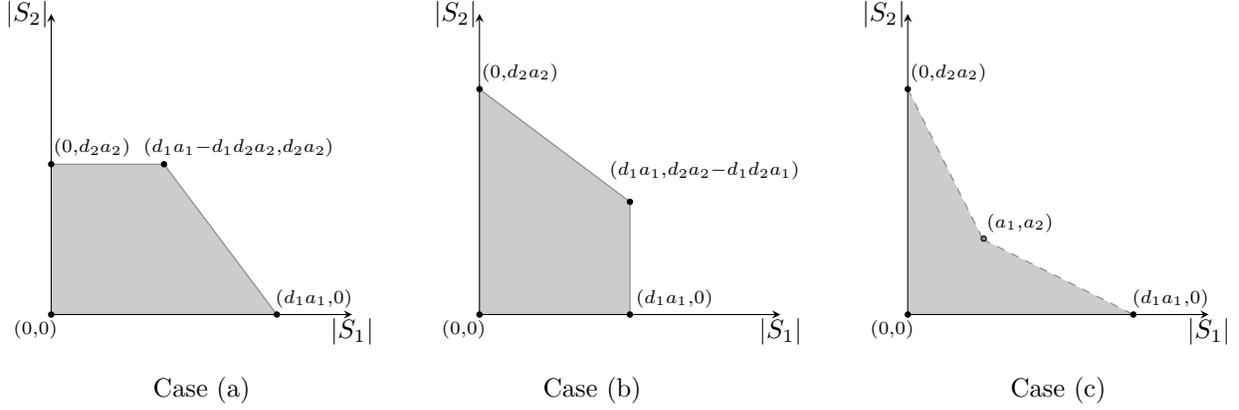
\begin{center}
  \begin{tikzpicture}
   \fill[gray!40!white] (0,0) -- (0,2) -- (1.5,2) -- (3,0) -- (0,0);
   \draw[color=gray] (0,2) -- (1.5,2);
   \draw[color=gray] (1.5,2) -- (3,0);
   \draw[-stealth] (0,0) -- (4,0);
   \draw[-stealth] (0,0) -- (0,4);
   \draw[fill=black] (0,0) circle (1pt);
   \draw[fill=black] (0,2) circle (1pt);
   \draw[fill=black] (1.5,2) circle (1pt);
   \draw[fill=black] (3,0) circle (1pt);
   \draw (-0.2,-0.2) node {${}_{(0,0)}$};
   \draw (3.5,0.22) node {${}_{(d_1a_1,0)}$};
   \draw (0.55,2.22) node {${}_{(0,d_2a_2)}$};
   \draw (2.5,2.22) node {${}_{(d_1a_1-d_1d_2a_2,d_2a_2)}$};
   \draw (4,-0.22) node {$|S_1|$};
   \draw (-0.32,4) node {$|S_2|$};
   \draw (2,-1) node {Case (a)};
  \end{tikzpicture}
  \hfill
  \begin{tikzpicture}
   \fill[gray!40!white] (0,0) -- (0,3) -- (2,1.5) -- (2,0) -- (0,0);
   \draw[color=gray] (0,3) -- (2,1.5);
   \draw[color=gray] (2,1.5) -- (2,0);
   \draw[-stealth] (0,0) -- (4,0);
   \draw[-stealth] (0,0) -- (0,4);
   \draw[fill=black] (0,0) circle (1pt);
   \draw[fill=black] (0,3) circle (1pt);
   \draw[fill=black] (2,1.5) circle (1pt);
   \draw[fill=black] (2,0) circle (1pt);
   \draw (-0.2,-0.2) node {${}_{(0,0)}$};
   \draw (2.6,0.22) node {${}_{(d_1a_1,0)}$};
   \draw (0.55,3.22) node {${}_{(0,d_2a_2)}$};
   \draw (3,1.92) node {${}_{(d_1a_1,d_2a_2-d_1d_2a_1)}$};
   \draw (4,-0.22) node {$|S_1|$};
   \draw (-0.32,4) node {$|S_2|$};
   \draw (1.5,-1) node {Case (b)};
  \end{tikzpicture}
  \hfill
  \begin{tikzpicture}
   \fill[gray!40!white] (0,0) -- (0,3) -- (1,1) -- (3,0) -- (0,0);
   \draw[color=gray,dashed] (0,3) -- (1,1);
   \draw[color=gray,dashed] (1,1) -- (3,0);
   \draw[-stealth] (0,0) -- (4,0);
   \draw[-stealth] (0,0) -- (0,4);
   \draw[fill=black] (0,0) circle (1pt);
   \draw[fill=black] (0,3) circle (1pt);
   \draw[fill=gray] (1.01,1.01) circle (1pt);
   \draw[fill=black] (3,0) circle (1pt);
   \draw (-0.2,-0.2) node {${}_{(0,0)}$};
   \draw (3.5,0.22) node {${}_{(d_1a_1,0)}$};
   \draw (0.55,3.22) node {${}_{(0,d_2a_2)}$};
   \draw (1.5,1.22) node {${}_{(a_1,a_2)}$};
   \draw (4,-0.22) node {$|S_1|$};
   \draw (-0.32,4) node {$|S_2|$};
   \draw (2,-1) node {Case (c)};
  \end{tikzpicture}
 \end{center}
 \begin{proof}
  There will be many cases to consider, we begin by handling the easier degenerate cases.\\
  
  \noindent $\bullet$ Suppose $a_1=a_2=0$.  Then the maximal Dyck path $D$ is a single vertex, i.e. $D_1=D_2=\emptyset$, so that $(|S_1|,|S_2|)=(0,0)$.  This verifies the result in this most degenerate of cases.\\
  
  \noindent $\bullet$ Suppose $a_1>a_2=0$.  Then the maximal Dyck path $D$ consists of $a_1$ consecutive horizontal edges and no vertical edges.  It follows that $(|S_1|,|S_2|)$ is contained in the segment $\big[(0,0),(d_1a_1,0)\big]$.  This verifies (a) in this degenerate case.\\
  
  \noindent $\bullet$ Suppose $a_2>a_1=0$.  Then the maximal Dyck path $D$ consists of $a_2$ consecutive vertical edges and no horizontal edges.  It follows that $(|S_1|,|S_2|)$ is contained in the segment $\big[(0,0),(0,d_2a_2)\big]$.  This verifies (b) in this degenerate case.\\
  
  \noindent $\bullet$ Suppose $a_1\ge d_2a_2>0$.  We first note that $|S_1|\le d_1a_1$ and $|S_2|\le d_2a_2$.  Thus it only remains to justify the boundary segment $\big[(d_1a_1,0),(d_1a_1-d_1d_2a_2,d_2a_2)\big]$.  In the present case we have $\rsh(S_2)=\emptyset$, thus according to Lemma~\ref{le:shadows} we have $|\sh(S_2)|=|S_2|$ and $S_1(h)=0$ for each $h\in\sh(S_2)$.  It follows that $|S_1|\le d_1a_1-d_1|S_2|$ and thus $(|S_1|,|S_2|)$ lies on or to the left of the segment.  This complete the proof of (a).\\
  
  \noindent $\bullet$ Suppose $a_2\ge d_1a_1>0$.  By a similar argument as above we have $|S_2|\le d_2a_2-d_2|S_1|$, from which (b) follows.\\
  
  \noindent $\bullet$ Suppose $0<a_1<d_2a_2$ and $0<a_2<d_1a_1$.  First note that $|S_1|\le d_1a_1$ and $|S_2|\le d_2a_2$.  Thus we only need to justify the boundary segments $\big[(d_1a_1,0),(a_2,a_1)\big]$ and $\big[(0,d_2a_2),(a_2,a_1)\big]$.  
 
 Following Lemma~\ref{le:grading bound} we will assume $0<|S_2|<a_1$ and justify the boundary segment $\big[(d_1a_1,0),(a_2,a_1)\big]$, the other boundary segment follows by a similar argument under the assumption $0<|S_1|<a_2$.  Note that the condition that $(|S_1|,|S_2|)$ lies strictly to the left of this segment is equivalent to the inequality
 \begin{equation}
  |S_1|<-\frac{d_1a_1-a_2}{a_1}|S_2|+d_1a_1.
 \end{equation}
 This translates into the inequality $a_1|S_1|<d_1a_1\big(a_1-|S_2|\big)+a_2|S_2|$, where we note that the quantity $a_1-|S_2|$ is exactly the number of edges outside the shadow of $S_2$.  Thus it suffices to show
 \begin{equation}
 \label{eq:shadow inequality}
  \text{$a_1|S_1|<a_2|S_2|$ for any $S_1\in\cC_{rs}(S_2)$.}
 \end{equation}
 
 Since $\sh(S_2)\subset D_1$ is a proper subset, Lemma~\ref{le:shadow containment} implies that $\sh(S_2)$ is a disjoint union of maximal local shadows each of which is a proper subset of $D_1$.  Let $\sh(v;S_2)$ denote such a maximal local shadow and write $D(v;S_2)=ev$ for the corresponding minimal shadow path.  Then, by Corollary~\ref{cor:remote_shadow_compatibility}, for any $h\in(ev)_1$ the local shadow path $D(h;S_1)$ is a proper subset of $ev$.  It follows that
 \begin{align}
 \label{eq:local inequality}
  a_1\Big|{S_1}|_{(ev)_1}\Big|\le a_1\Big(|(ev)_2|-1\Big)<a_2|(ev)_1|=a_2\Big|{S_2}|_{(ev)_2}\Big|,
 \end{align}
 where the second inequality follows from Corollary~\ref{cor:slopes} and the last equality is the condition \eqref{eq:compatibility2} for the edges $e$ and $v$.  
 
 Since $S_1(h)=0$ for any horizontal edge $h$ outside the shadow of $S_2$ and $S_2(v)=0$ for any   vertical edge outside of a minimal shadow path, \eqref{eq:local inequality} implies \eqref{eq:shadow inequality}.  This completes the proof of the final case and thus concludes the proof of Proposition~\ref{prop:supports}.
 \end{proof}

 \section{Appendix: Multinomial Coefficients}\label{sec:multinomial}
 In this section we collect certain lesser-known analogues for multinomial coefficients of well-known identities involving binomial coefficients.
 
 The binomial coefficients describe the coefficients when expanding powers of binomials, generalizing this we define multinomial coefficients ${n\choose k_1,\ldots,k_r}$ by
  \begin{equation}\label{eq:multinomial definition}
   (x_1+\cdots+x_r)^n=\sum\limits_{(k_1,\ldots,k_r)\partition n} {n\choose k_1,\ldots,k_r} x_1^{k_1}\cdots x_r^{k_r},
  \end{equation}
  where $(k_1,\ldots,k_r)\partition n$ denotes a partition of the positive integer $n$ into $r$ positive parts, i.e. $k_1,\ldots,k_r\in\ZZ_{\ge0}$ with $k_1+\cdots+k_r=n$.  By convention, for $n\ge0$ we will take ${n\choose k_1,\ldots,k_r}=0$ whenever $k_i<0$ for some $i$.  We will discuss $n<0$ below.
 
 To begin, the multinomial coefficients satisfy an analogue of the Pascal identity for binomial coefficients.
 \begin{lemma}\label{le:multinomial Pascal identity}
  Suppose $n\ge1$.  For any partition $(k_1,\ldots,k_r)\partition n$ we have the ``Pascal identity"
  \[{n\choose k_1,\ldots,k_r}={n-1\choose k_1-1,\ldots,k_r}+\cdots+{n-1\choose k_1,\ldots,k_r-1}.\]
 \end{lemma}
 \begin{proof}
  We apply the definition of ${n\choose k_1,\ldots,k_r}$ and expand in two different ways:
  \begin{align*}
   (x_1+\cdots+x_r)^n
   &=(x_1+\cdots+x_r)(x_1+\cdots+x_r)^{n-1}\\
   &=(x_1+\cdots+x_r)\sum\limits_{(k_1,\ldots,k_r)\partition n-1} {n-1\choose k_1,\ldots,k_r} x_1^{k_1}\cdots x_r^{k_r}\\
   &=\sum\limits_{(k_1,\ldots,k_r)\partition n} \left[{n-1\choose k_1-1,\ldots,k_r}+\cdots+{n-1\choose k_1,\ldots,k_r-1}\right] x_1^{k_1}\cdots x_r^{k_r}.
  \end{align*}
  The coefficient of $x_1^{k_1}\cdots x_r^{k_r}$ in $(x_1+\cdots+x_r)^n$ is thus described by ${n-1\choose k_1-1,\ldots,k_r}+\cdots+{n-1\choose k_1,\ldots,k_r-1}$ and also ${n\choose k_1,\ldots,k_r}$, so these must be equal.
 \end{proof}
 
 Using this recursive description of the multinomial coefficients one may derive an expression in terms of factorials, again analogous to the well-known formula for binomial coefficients.
 \begin{lemma}\label{le:multinomial factorial identity}
  Suppose $n\ge0$.  For any partition $(k_1,\ldots,k_r)\partition n$ we have
  \[{n\choose k_1,\ldots,k_r}=\frac{n!}{k_1!\cdots k_r!}.\]
 \end{lemma}
 \begin{proof}
  We will work by induction, when $n=0$ there is only the trivial partition $(0,\ldots,0)\partition_r 0$, where $\partition_r$ denotes a partition into $r$ parts.  On both sides of the desired equality we have $1$, establishing the base of the induction.
  
  Suppose $n>0$ and consider a partition $(k_1,\ldots,k_r)\partition n$.  Using Lemma~\ref{le:multinomial Pascal identity} and the induction hypothesis we have
  \begin{align*}
   {n\choose k_1,\ldots,k_r}
   &={n-1\choose k_1-1,\ldots,k_r}+\cdots+{n-1\choose k_1,\ldots,k_r-1}\\
   &=k_1\frac{(n-1)!}{k_1!\cdots k_r!}+\cdots+k_r\frac{(n-1)!}{k_1!\cdots k_r!}\\
   &=(k_1+\cdots+k_r)\frac{(n-1)!}{k_1!\cdots k_r!}\\
   &=\frac{n!}{k_1!\cdots k_r!},\\
  \end{align*}
  where we remark for the reader that there is no conflict in the second equality with our convention for nonnegative $n$ that ${n\choose k_1,\ldots,k_r}=0$ whenever $k_i<0$ for some $i$, for example if $k_1=0$ then both ${n-1\choose k_1-1,\ldots,k_r}$ and $k_1\frac{(n-1)!}{k_1!\cdots k_r!}$ are zero.
 \end{proof}
 
 \begin{corollary}\label{cor:binomial and multinomial coefficients}
  For $n,k\ge0$ and $(k_1,\ldots,k_r)\partition k$ we have the following identity relating binomial coefficients and multinomial coefficients
  \[{n\choose k}{k\choose k_1,\ldots,k_r}={n\choose n-k,k_1,\ldots,k_r}.\]
 \end{corollary}
 \begin{proof}
  This is an immediate consequence of Lemma~\ref{le:multinomial factorial identity} and the well-known identity ${n\choose k}=\frac{n!}{k!(n-k)!}$.
 \end{proof}
 Recall that we can make sense of ${n\choose k}$ for $k\ge0$ and $n\in\ZZ$, i.e. for $n\ge1$ we have the equality ${-n\choose k}=(-1)^k{n+k-1\choose k}$.  Thus we can make sense of ${n\choose k_0,k_1,\ldots,k_r}$ for $n,k_0\in\ZZ$ by applying Corollary~\ref{cor:binomial and multinomial coefficients}.  More precisely, when $n,k_0<0$ we have
 \[{n\choose k_0,k_1,\ldots,k_r}={n\choose n-k_0}{n-k_0\choose k_1,\ldots,k_r}=(-1)^{n-k_0}{-k_0-1\choose n-k_0}{n-k_0\choose k_1,\ldots,k_r},\]
 where $-k_0-1\ge0$.  By our conventions, the binomial coefficient ${-k_0-1\choose n-k_0}$ is zero for $n-k_0<0$, in particular the multinomial coefficient ${n\choose k_0,k_1,\ldots,k_r}$ is zero whenever $n<k_0$.

 Our primary goal in this section is to understand how to use multinomial coefficients to expand as power series the negative powers of a polynomial in a single variable.  The following lemma is slightly more general but will serve our purposes here.
 \begin{lemma}\label{le:negative exponent expansion}
  For any $n>0$ and $c\in\kk^\times$ we have
  \begin{equation}\label{eq:negative exponent expansion}
   (c+x_1+\cdots+x_r)^{-n}=\sum\limits_{k\ge0}\sum\limits_{(k_1,\ldots,k_r)\partition k}(-1)^k{n+k-1\choose n-1,k_1,\ldots, k_r} c^{-n-k}x_1^{k_1}\cdots x_r^{k_r}.
  \end{equation}
 \end{lemma}
 \begin{proof}
  We apply the standard binomial formula with negative exponents and then the definition of the multinomial coefficients to get
  \begin{align*}
   (c+x_1+\cdots+x_r)^{-n}
   &=c^{-n}\big(1+(x_1+\cdots+x_r)/c\big)^{-n}\\
   &=\sum\limits_{k\ge0}(-1)^k{n+k-1\choose k}c^{-n-k}(x_1+\cdots+x_r)^k\\
   &=\sum\limits_{k\ge0}\sum\limits_{(k_1,\ldots,k_r)\partition k}(-1)^k{n+k-1\choose k}{k\choose k_1,\ldots,k_r} c^{-n-k}x_1^{k_1}\cdots x_r^{k_r}\\
   &=\sum\limits_{k\ge0}\sum\limits_{(k_1,\ldots,k_r)\partition k}(-1)^k{n+k-1\choose n-1,k_1,\ldots, k_r} c^{-n-k}x_1^{k_1}\cdots x_r^{k_r},
  \end{align*}
  where the last equality follows from Corollary~\ref{cor:binomial and multinomial coefficients}.
 \end{proof}
 \begin{remark}
  Notice that the result of Lemma~\ref{le:negative exponent expansion} makes sense for any $n\in\ZZ$.  Indeed, applying the definition of the standard (positive) multinomial coefficients to $(c+x_1+\cdots+x_r)^n$ for $n\ge0$ we get
  \begin{align*}
   (c+x_1+\cdots+x_r)^n
   &=\sum\limits_{(k_0,k_1,\ldots,k_r)\partition n} {n\choose k_0,k_1,\ldots, k_r} c^{k_0}x_1^{k_1}\cdots x_r^{k_r}\\
   &=\sum\limits_{k=0}^n\sum\limits_{(k_1,\ldots,k_r)\partition k} {n\choose n-k,k_1,\ldots, k_r} c^{n-k}x_1^{k_1}\cdots x_r^{k_r}\\
   &=\sum\limits_{k\ge0}\sum\limits_{(k_1,\ldots,k_r)\partition k} {n\choose k}{k\choose k_1,\ldots, k_r} c^{n-k}x_1^{k_1}\cdots x_r^{k_r}\\
   &=\sum\limits_{k\ge0}\sum\limits_{(k_1,\ldots,k_r)\partition k}(-1)^k{-n+k-1\choose k}{k\choose k_1,\ldots, k_r} c^{n-k}x_1^{k_1}\cdots x_r^{k_r}\\
   &=\sum\limits_{k\ge0}\sum\limits_{(k_1,\ldots,k_r)\partition k}(-1)^k{-n+k-1\choose -n-1,k_1,\ldots, k_r} c^{n-k}x_1^{k_1}\cdots x_r^{k_r},\\
  \end{align*}
  which agrees with \eqref{eq:negative exponent expansion} as claimed.
 \end{remark}
 
 To complete our discussion of multinomial coefficients consider a polynomial $P(z)\in\kk[z]$ of degree $d$ and write $P(z)=\rho_0+\rho_1z+\cdots+\rho_dz^d$.
 \begin{corollary}\label{cor:polynomial negative power}
  For $n\in\ZZ$ we have
  \[P(z)^n=\sum\limits_{k\ge0}\sum\limits_{(k_1,\ldots,k_d)\partition k} (-1)^k \rho_0^{n-k}\rho_1^{k_1}\cdots \rho_d^{k_d}{-n+k-1\choose -n-1,k_1,\ldots, k_d} z^{k_1+2k_2+\cdots+dk_d}.\]
 \end{corollary}

\end{document}